\documentclass[12pt]{amsart}
\usepackage{amssymb}
\usepackage{latexsym}
\usepackage{amsmath}
\usepackage{amsthm}
\usepackage{amsfonts}
\usepackage{color}

\theoremstyle{plain}
\newtheorem{theorem}{Theorem}
\theoremstyle{plain}
\newtheorem{lemma}{Lemma}
\theoremstyle{plain}
\newtheorem{corollary}{Corollary}
\theoremstyle{plain}
\newtheorem{propos}{Proposition}
\theoremstyle{plain}
\newtheorem{definition}{Definition}
\theoremstyle{plain}
\newtheorem{remark}{Remark}
\newtheorem{example}{Example}

\title
[Dilations and translations]
{Sequences of dilations and translations equivalent to the Haar system in $L^p$-spaces}

\author[S.V. Astashkin]{Sergey V. Astashkin}
\address{Department of Mathematics,
Samara National Research University, Moskovskoye shosse 34, 443086,
Samara, Russia}
\email{astash56@mail.ru}

\thanks{The work of the first author was supported by the Ministry of Science and Higher Education of the Russian Federation, project 1.470.2016/1.4 and by the RFBR grant 18-01-00414.}
\thanks{The work of the second author was supported by SEMC ``Mathematics of Future Technologies''.}

\author[P.~A.~Terekhin]{Pavel A. Terekhin}
\address{Department of Mechanics and Mathematics,
Saratov State University, Astrakhanskaya Str 83, 410012,
Saratov, Russia}
\email{terekhinpa@mail.ru}

\date{\today}


\begin{document}

\begin{abstract}
Let $f=\sum_{k=0}^{\infty}c_kh_{2^k}$, where $\{h_n\}$ is the classical Haar system, $c_k\in\mathbb{C}$. Given a $p\in (1,\infty)$, we find the sharp conditions, under which the sequence $\{f_n\}_{n=1}^\infty$ of dilations and translations of $f$ is a basis in the space $L^p[0,1]$, equivalent to $\{h_n\}_{n=1}^\infty$. The results obtained depend substantially on whether $p\ge 2$ or $1<p<2$ and include as the endpoints of the $L_p$-scale the spaces $BMO_d$ and $H_d^1$. The proofs are based on an appropriate splitting the set of positive integers $\mathbb{N}=\cup_{d=1}^\infty N_d$ so that the equivalence of $\{f_n\}_{n=1}^\infty$ to the Haar system in $L_p$ would be ensured by the fact that  $\{f_n\}_{n\in N_d}$ is a basis in the subspace $[h_{m},m\in N_d]_{L_p}$,  equivalent to the Haar subsequence $\{h_n\}_{n\in N_d}$ for every $d=1,2,\dots$.
\end{abstract}

\maketitle


\section{Introduction and main results}
\label{Intro}

Many issues related to the geometry of the $L^p$-spaces and more generally rearrangement invariant spaces are closely connected with properties of the classical Haar system $\{h_n\}_{n=0}^{\infty}$ (see e.g. \cite{JMST,KS,LT2,M,NS}). It is well known that $\{h_n\}$ is an unconditional basis in $L^p$ if and only if $1<p<\infty$ (see for different proofs of this fundamental fact  \cite[Theorem~2.c.5]{LT2}, \cite[Theorems~3.3.8 and 3.3.10]{KS}, \cite[Theorem~1.1.5]{M}, \cite[Theorem~5.5]{NS}). It is worth mentioning also that in the case $1<p<\infty$ Gamlen and Gaudet  \cite{GG} have found the isomorphic types of subspaces spanned by subsequences of the Haar system. Specifically, they have showed that each subsequence $\{h_{n_i}\}\subset \{h_n\}$ spans a subspace in $L^p[0,1]$, which is isomorphic either to $\ell^p$ or to $L^p[0,1]$.

Observe that the Haar system is the simplest and most important  representative of a much wider class of systems of dilations and translations. Let a function $f\in L^1[0,1]$ have zero mean value, i.e., $\int_0^1f(t)\,dt=0$.
\begin{definition}\label{d.aff}
The {\it sequence of dilations and translations} of $f$ (or the {\it Haar affine system generated by $f$}) consists of the functions
$$
f_n(t):=\begin{cases}
f(2^kt-j), & \mbox{if} \,\,\, t\in(\tfrac{j}{2^k},\tfrac{j+1}{2^k}],\\
0, & \mbox{otherwise},
\end{cases}
$$
where $n=2^k+j$, $k=0,1,\dots$ and $j=0,\dots,2^k-1$.
\end{definition}

In particular, if we take for $f$ the function $h(t):=h_1(t)=\chi_{(0,\frac12]}-\chi_{(\frac12,1]}$, we get the Haar system $\{h_n\}_{n=1}^{\infty}$, normalized in $L^{\infty}$ (without the function $h_0(t)=1$).

Let $1<p<\infty$. The main aim of this paper is the identification of conditions, under which the sequence of dilations and translations $\{f_{n}\}_{n=1}^\infty$ of a mean zero function $f$ is equivalent in the space $L^p$ to the Haar system.
The main idea behind is an appropriate splitting the set of positive integers $\mathbb{N}=\cup_{d=1}^\infty N_d$ so that the equivalence of $\{f_{n}\}_{n=1}^\infty$ to the Haar system would be ensured by the fact that  $\{f_{n}\}_{n\in N_d}$ is a basis in the subspace $[h_{m},m\in N_d]$,  equivalent to the Haar subsequence $\{h_{n}\}_{n\in N_d}$ for every $d=1,2,\dots$
This makes it possible to reduce the verification of the basisness and equivalence of a given system of dilations and translations to the Haar system to a similar checking for sequences $\{f_{n}\}_{n\in N_d}$, $d=1,2,\dots$, demonstrating thereby the interesting phenomenon that the certain properties of sequences of such a type follow from those of some their "special"\:subsequences.

One can see that this approach includes the additional requirement of the containment of the subsequence $\{f_{n}\}_{n\in N_d}$ into the subspace $[h_{m},m\in N_d]$ for each $d=1,2,\dots$ This justifies our further choice of the class of functions, whose sequences of dilations and translations we will consider.

Observe that the Haar function $h$ belongs to the following class of step-functions:
\begin{equation}\label{Haar chaos}
f(t)=\sum_{k=0}^\infty f(\tfrac{1}{2^k})\chi_{(\tfrac{1}{2^{k+1}},\tfrac{1}{2^k}]}(t), \quad 0<t\le 1.
\end{equation}
It is plain that such a mean zero function $f\in L^p$, $1<p<\infty$, has the Fourier-Haar expansion of the form
\begin{equation}\label{F-H expansion}
f=\sum_{k=0}^{\infty}c_kh_{2^k}, \quad\mbox{where}\; c_k=2^k(f,h_{2^k})
\end{equation}
(in what follows, for any measurable functions $u$ and $v$ we set $(u,v):=\int_0^1 u(t)v(t)\,dt$, provided if this integral is well-defined). Moreover, for every $d\in\mathbb{N}$ we put
$$
N_d:=\{n\in\mathbb{N}:\,n=2^{i_1}+\ldots+2^{i_d},\;\mbox{where}\;i_1>\ldots>i_d\ge0,  i_j\in{\mathbb Z}_+\}.$$
Then, by the definition of dilations and translations of a function, for any $n\in N_d$ 
$$
f_{n}=\sum_{k=0}^{\infty}c_kh_{2^kn}.
$$
Since $2^kn=2^{i_1+k}+\ldots+2^{i_d+k}\in N_d$ as well, it follows that $f_{n}\in[h_{m},m\in N_d]_{L^p}$ (see also Lemma~\ref{l.1}).
Thus, we have $[f_{n},n\in N_d]_{L^p}\subset [h_{n},n\in N_d]_{L^p}$ for every $d=1,2,\dots$. Observe, however, that the stronger condition $[f_{n},n\in N_d]_{L^p}=[h_{n},n\in N_d]_{L^p}$ need not to be fulfilled, in general (see Lemma~\ref{l.4} and Remark~\ref{rem: vanish}).

Following \cite[Definition~4]{AT2018}, we will call each function $x$ of the form
$$
x=\sum_{n\in N_d}\xi_nh_{n}, \quad \xi_n\in\mathbb{C},
$$
a {\it Haar chaos of order $d$} and denote by $\mathcal{H}_{ch}^d$ the set of all such functions.\footnote{We follow here the common terminology related to chaoses in independent random variables (cf. \cite[p.~147]{KwW}).} Observe that the latter series converges a.e. on $[0,1]$ and hence such a function $x$ is well-defined. Clearly, $\mathcal{H}_{ch}^d \cap L^p=[h_{n},n\in N_d]_{L^p}$, $d=1,2,\dots$.

The study of sequences of dilations and translations generated by Haar chaoses was initiated in the paper \cite{AT2018}. In particular, it was proved there (see \cite[Corollary~4]{AT2018}) that the sequence of dilations and translations of a first-order Haar chaos $f$, defined by \eqref{F-H expansion}, is equivalent in $L^p$ to the Haar system for all $1<p<\infty$ if and only if the generating function
$$
\hat{f}(z)=\sum_{k=0}^{\infty}c_kz^k,
$$
corresponding to the sequence $(c_k)_{k=0}^\infty$, is analytic and does not vanish in the unit disk $\mathbb{D}$. Moreover, from Theorem~7 of the paper \cite{AT2018a}, devoted to studying affine Walsh-type systems, it follows easily that, for every $1<p<\infty$, the sequence of dilations and translations of such a function $f$ is a basis in $L^q$ for $1<q<p$ if and only if the function $\hat{f}(z)$ is analytic and does not vanish in the smaller disk $\mathbb{D}_{2^{-1/p}}$. However, the problem of identifying the sharp conditions, ensuring the equivalence of  the sequence of dilations and translations of $f$ to the Haar system, for the certain value of $p\in (1,\infty)$, was remained open. Here, we go somewhat deeper and address this  question, showing that the answer depends substantially on whether $p\ge 2$ or $1<p<2$, which reflects the well-known fact that geometrical properties of the $L^p$-spaces in these cases are different.\footnote{Comparing to that, the above-mentioned criterion from the paper \cite{GG} to determine the isomorphic type of the subspace, spanned in $L^p$ by a given subsequence $\{h_{n_i}\}\subset \{h_n\}$, is independent of $p$.}

Let us mention briefly another direction of research related to studying  systems of dilations and translations, namely, the representation of functions of various function spaces. It was started as long ago as in 1941 by Men'shov in the case of the trigonometric system \cite{Menchoff}, and then was continued by Talalyan \cite{Tal}, which has introduced explicitly the notion of representing system (see also the survey \cite{Ul'yanov}).  Concerning the study of representing systems of dilations and translations in the $L_p$-spaces see the paper \cite{FO95} (and references therein) as well the recent work \cite{AT2019}, where this problem is considered in a more general setting of rearrangement invariant spaces.

A crucial role in the proofs will be played by properties of the generating function $\hat{f}(z)$, corresponding to the sequence of Fourier-Haar coefficients of a given first-order Haar chaos $f$. Therefore, before stating our main results, we need to introduce some suitable spaces of analytic functions.

Let $1\le p\le\infty$, $R>0$, and let $A_p^+(\mathbb{D}_R)$ be the Banach space of all analytic functions $u(z)=\sum_{k=0}^{\infty}a_kz^k$ in the disk $\mathbb{D}_R=\{z\in\mathbb{C}:|z|<R\}$, equipped with the norm
$$
\|u\|_{A_p^+(\mathbb{D}_R)}:=\|(a_kR^k)\|_{\ell^p}
=\biggl(\sum_{k=0}^{\infty}(|a_k|R^k)^p\biggr)^{1/p}.
$$
We say that a function $u\in A_p^+(\mathbb{D}_R)$ is a {\it $p$-multiplier}  if $uv\in A_p^+(\mathbb{D}_R)$ for all $v\in A_p^+(\mathbb{D}_R)$.
The space $M_p^+(\mathbb{D}_R)$, consisting of all $p$-multipliers and endowed with the norm
$$
\|u\|_{M_p^+(\mathbb{D}_R)}:=\sup_{\|v\|_{A_p^+(\mathbb{D}_R)}\le1}\|uv\|_{A_p^+(\mathbb{D}_R)},
$$
is a Banach algebra with respect to the usual multiplication of functions.


As it will follow from our proofs, the key condition of invertibility of the function $\hat{f}(z)$ in the algebra $M_p^+(\mathbb{D}_{2^{-1/p}})$ means the equivalence of each subsystem $\{f_n\}_{n\in N_d}$, $d\in\mathbb{N}$, of the system of dilations and translations of a function $f\in \mathcal{H}_{ch}^d \cap L^p$ to the corresponding Haar subsystem (see the very beginning of the proofs of Propositions~\ref{l.L2} and \ref{l.Lpg2}).
So, the main question can be stated as follows: Does  this fact ensure the equivalence of the whole system $\{f_n\}_{n=1}^\infty$ to the Haar system? As was already said above, the results show that the answer depends largely on the value of $p$.



\begin{theorem}\label{t.1}
Let $p\ge2$, and let $f\in L^p$ be a Haar chaos of order $1$.
The following conditions are equivalent:

(i) the system of dilations and translations of $f$ is a basis in $L^p$ that is equivalent to the Haar system;

(ii) for every $d=1,2,\dots$ the subsystem $\{f_{n}\}_{n\in N_d}$ is a basis in the subspace $\mathcal{H}_{ch}^d\cap L^p$,  equivalent to the Haar subsystem $\{h_{n}\}_{n\in N_d}$;

(iii) the subsystem $\{f_{2^k}\}$ is a basis in the subspace $\mathcal{H}_{ch}^1\cap L^p$, equivalent to the Haar subsystem $\{h_{2^k}\}$;

(iv) the functions $\hat{f}(z)$ and $1/\hat{f}(z)$ belong to the space $M_p^+(\mathbb{D}_{2^{-1/p}})$.
\end{theorem}

We establish also an analogue of Theorem \ref{t.1} in the limiting case $p=\infty$ for the dyadic space $BMO_d$ of functions with bounded mean oscillation    (see Corollaries \ref{c.BMOd1} and \ref{c.BMOd2}).

Somewhat surprisingly that the assertions of Theorem \ref{t.1} cannot be extended to the range $1<p<2$. Moreover, as the next result shows, the answer to the question if the system of dilations and translations of $f$ is equivalent or not to the Haar system in $L^p$ in this case does not depend on the value of $p\in (1,2)$.

\begin{theorem}\label{t.2}
Let $f\in L^1$ be a Haar chaos of order $1$.
The following conditions are equivalent:

(a) there is $p\in (1,2]$ such that $f\in L^p$ and the system of dilations and translations of $f$ is equivalent in $L^p$  to the Haar system;

(b) $f\in L^2$ and for each $p\in (1,2]$ the system of dilations and translations of $f$ is equivalent in $L^p$  to the Haar system.
\end{theorem}

An analogous result holds also in the limiting case $p=1$ provided that we consider the space $H_d^1$ that is predual for $BMO_d$ (see Corollary \ref{c.Hd2}).

An important consequence of Theorem \ref{t.2} is the existence of a rather wide class of functions $f\in L^p$, $1<p<2$, such that the corresponding system of dilations and translations is a basis in $L^p$, non-equivalent to the Haar system. Observe that, by using subtle properties of analytic functions related to their behaviour on the boundary of the disk $\mathbb{D}_{2^{-1/p}}$, such a basis can be constructed also in the case $p\ge2$. In contrast to that, for $p<2$ the non-equivalence of the system of dilations and translations of $f$ to the Haar system in $L^p$ is ensured just by the fact that the radius of analyticity of $\hat{f}(z)$ or $1/\hat{f}(z)$ is less than the critical radius $2^{-1/2}$.
More precisely, the following result holds.

\begin{theorem}\label{t.3}
Let $1<p\le2$, and let $f\in L^1$ be a Haar chaos of order $1$.

Suppose the function $\hat{f}(z)$ is analytic and does not vanish in the disk $\mathbb{D}_{2^{-1/p}}$, but at least one of the functions $\hat{f}(z)$ or $1/\hat{f}(z)$ is not bounded on $\mathbb{D}_{2^{-1/2}}$.

Then the system of dilations and translations of $f$ is a basis in $L^q$ for every  $1<q<p$ that is non-equivalent to the Haar system.
\end{theorem}

Thus, the condition of invertibility of the function $\hat{f}(z)$ in the algebra of multipliers turns, in the case $p<2$, weaker than the equivalence in $L^p$ of the system of dilations and translations of a first-order Haar chaos $f$ to the Haar system. To clarify the meaning of the latter condition, we introduce the following notion.


The system of dilations and translations of a function $f$ is said to be a {\it basis in chaoses} in $L^p$, whenever for all mean zero $x\in L^p$ we have
\begin{equation}
\label{exp via chaoses}
x=\sum_{d=1}^{\infty}\sum_{n\in N_d}(x,g^n)f_n\;\;\mbox{(convergence in}\;L^p),
\end{equation}
where $(g^n)$ is the biorthogonal system to the system $(f_n)$ (its explicit form will be given in Section~\ref{dual}).


\begin{theorem}\label{t.4}
Let $1<p<\infty$. Suppose $f\in L^p$ is a Haar chaos of order $1$ such that both functions $\hat{f}(z)$ and $1/\hat{f}(z)$ belong to the space $M_p^+(\mathbb{D}_{2^{-1/p}})$. Then,  the system of dilations and translations of $f$ is a basis in chaoses in $L^p$.
\end{theorem}

In the case $p\ge2$ Theorem \ref{t.4} is clearly an immediate consequence of Theorem \ref{t.1}, which implies that the system of dilations and translations of $f$ is an unconditional basis of the space $L^p$.

Throughout, we consider all function spaces over the complex field. Moreover, we write $F\lesssim G$ if the inequality $F\le CG$ holds with a constant $C>0$ independent of all or a part of arguments of quantities (norms) $F$ and $G$. The notation $F\simeq G$ means that $F\lesssim G$ and $G\lesssim F$.

\section{Preliminaries and auxiliary results}

\subsection{Operator generation of dilations and translations}
\label{generation}

Let
$$
\mathbb{A}=\bigcup_{k=0}^{\infty}\{0,1\}^k,
$$
that is, the family $\mathbb{A}$ consists of all multi-indices $\alpha=(\alpha_1,\dots,\alpha_k)$, $k=0,1,2,\dots$, such that $\alpha_\nu=0$ or $1$, $\nu=1,\dots,k$
(for $k=0$ we obtain the empty set $\varnothing$, which will be also assumed to be an element of $\mathbb{A}$).

Denote by $|\alpha|(=k)$ the length of $\alpha=(\alpha_1,\dots,\alpha_k)\in\mathbb{A}$ and by
$\alpha\beta$ the concatenation of $\alpha=(\alpha_1,\dots,\alpha_k)$ and $\beta=(\beta_1,\dots,\beta_l)$, i.e., the multi-index
$(\alpha_1,\dots,\alpha_k,\beta_1,\dots,\beta_l)$.
The set $\mathbb{A}$ can be treated as a free semigroup with two generators $\{0\}$ and $\{1\}$ with respect to the concatenation or as an infinite complete binary tree. To each $\alpha=(\alpha_1,\dots,\alpha_k)\in\mathbb{A}$ we assign the dyadic subinterval $I_{\alpha}$ of the interval $I=(0,1]$ by
$$
I_{\alpha}=(\tfrac{j}{2^k},\tfrac{j+1}{2^k}], \quad \mbox{where}\;k=|\alpha|\;\mbox{and}\; j=\sum_{\nu=1}^k\alpha_{\nu}2^{k-\nu}.
$$
Clearly, $I_{\alpha}\subset I_{\beta}$ if and only if $\alpha=\beta\gamma$ for some $\gamma\in\mathbb{A}$. Moreover, a family of intervals $\{I_{\alpha}\}_{\alpha\in D}$, $D\subset\mathbb{A}$, is disjoint if and only if from the equality $\alpha=\beta\gamma$, where $\alpha,\beta\in D$ and $\gamma\in\mathbb{A}$,
it follows that $\alpha=\beta$. The multi-index $\alpha$ of length $|\alpha|=k$ such that $\alpha_{\nu}=0$ for all $\nu=1,\dots,k$ will be denoted  by $0_k$. If $d\in\mathbb{N}$, then a multi-index $\alpha$, which contains precisely $d-1$ "ones"\:, can be written as $\alpha=(0_{k_1},1,0_{k_2},1,\dots,1,0_{k_d})$, where $k_1,\dots,k_d\ge0$.


Let $t=(0_{k_1},1,0_{k_2},1,\dots,0_{k_d},1,\dots)$ be the binary notation of a real number $t\in (0,1]$, or equivalently,
$$
t=\sum_{d=1}^{\infty}\frac{1}{2^{k_1+\ldots+k_d+d}}.
$$
Then, we have $I_{\alpha}\ni t$ for a multi-index $\alpha\in\mathbb{A}$ if and only if $\alpha=(0_{k_1},1,0_{k_2},1,\dots,0_{k_{d-1}},1,0_k)$ for some $d\in\mathbb{N}$ and $k=0,\dots,k_d$. Sometimes, for notational convenience, we write $(\xi(\alpha):\alpha\in\mathbb{A})$ instead of $(\xi_{\alpha}:\alpha\in\mathbb{A})$ (resp. $\xi_{k_1,\dots,k_d}$ instead of $\xi(0_{k_1},1,\dots,1,0_{k_d})$ for a family indexed by a subset of $\mathbb{A}$ of multi-indices containing precisely $d-1$ "ones").

Further, for a function $f(t)$, $t\in I$, we let
$$
V_0f(t)=\begin{cases}
f(2t), & t\in(0,\frac12],\\
0, & t\in(\frac12,1],
\end{cases}
\qquad
V_1f(t)=\begin{cases}
0, & t\in(0,\frac12],\\
f(2t-1), & t\in(\frac12,1].
\end{cases}
$$
Then, the system of dilations of translations of $f$ can be generated by the products
$$
V^{\alpha}=V_{\alpha_1}\ldots V_{\alpha_k}, \qquad \alpha=(\alpha_1,\dots,\alpha_k)\in\mathbb{A},
$$
of the operators $V_0$ and $V_1$ (in particular, $V^\varnothing$ is the identity $I$) as follows:
$$
V^{\alpha}f(t)=f_{k,j}(t), \qquad k=|\alpha|, \qquad j=\sum_{\nu=1}^k\alpha_{\nu}2^{k-\nu}.
$$
Note that the connection between this numbering and the natural one of the system of dilations and translations of $f$ (see Section~\ref{Intro}) is given by the binary expansion of positive integers, i.e., by the formula $f_n=V^{\alpha}f$, where $n=2^k+\sum_{\nu=1}^k\alpha_{\nu}2^{k-\nu}$, $n=1,2,\dots$.

\subsection{Haar multishift}

It is worthwhile to discuss the analogy between the operators $V_0$, $V_1$ and the well-known Halmos abstract characterization of the one-sided shift operator as an isometry $S$ in a Hilbert space $\mathcal{U}$, for which there is a vector $e\in\mathcal{U}$ such that the vectors $S^ke$, $k=0,1,\dots$, form an orthonormal basis in $\mathcal{U}$ (cf. ~\cite{H}).
As a canonical representation of this shift is considered often the multiplication operator $Su(z)=zu(z)$ in the Hardy space $H^2(\mathbb{D})$.
The following notion of multishift, introduced in \cite{T}, can be treated as an extension of the above notion to the case of two non-commuting operators.

\begin{definition}\label{d.mshift}
A pair of isometries $S_0$ and $S_1$ of a Hilbert space $\mathcal{U}$ is said to be a {\it multishift} whenever there is a vector $e\in\mathcal{U}$ such that the system $\{S^{\alpha}e\}_{\alpha\in\mathbb{A}}$ is an orthonormal basis in $\mathcal{U}$ (as in Section~\ref{generation}, $S^{\alpha}=S_{\alpha_1}\ldots S_{\alpha_k}$ is the operator product, corresponding to a multi-index $\alpha\in\mathbb{A}$).
\end{definition}

One can readily see that the operators $\sqrt{2}V_0$ and $\sqrt{2}V_1$ form a multishift in the subspace of $L^2[0,1]$ consisting of all mean zero functions, and the Haar system $\{2^{|\alpha|/2}V^{\alpha}h\}_{\alpha\in\mathbb{A}}$ is the corresponding orthonormal basis. By this reason, the pair of operators $V_0$ and $ V_1 $ will be called the {\it Haar multishift}.

Given a function $f$, we define the linear operator $T_f$ by
$$
T_fV^{\alpha}h=V^{\alpha}f, \qquad \alpha\in\mathbb{A}.
$$
In what follows, we repeatedly use the fact that the operator $T_f$ commutes with the operators $V_0$ and $V_1$, that is, for each mean zero Haar polynomial $x$ we have
$$
TV_0x=V_0Tx\;\;\mbox{and}\;\;TV_1x=V_1Tx.
$$
Conversely, every operator $T$, defined on the linear manifold of all mean zero Haar polynomials and commuting with $V_0$ and $V_1$, can be represented as $T_f$, where $f:=Th$ (indeed, $Th_{\alpha}=TV^{\alpha}h=V^{\alpha}Th=f_{\alpha}$, $\alpha\in\mathbb{A}$).

Of fundamental importance is the problem of a description of the commutant $c(V_0,V_1)_{L^p}$ of the Haar multishift, which is defined as the set of mean zero functions $f\in L^p$ such that the operator $T_f$ is bounded on $L^p$. More precisely, we say that the operator $T_f$ is bounded on the space $L^p$ if it is bounded on the subspace of $L^p$ of all mean zero functions; by $\|T_f\|_{L^p}$ we will denote its usual operator norm on this subspace.


Recall that every operator $T$ commuting with the shift operator $S$ in the Hardy space $H^2(\mathbb{D})$ has the form $T_u$, where $T_uv(z)=u(z)v(z)$, and it is bounded in $H^2(\mathbb{D})$ if and only if $u\in H^{\infty}(\mathbb{D})$ (see e.g. \cite[Chapter~17, Problem~147]{H2}). In contrast to this classical situation, a comprehensive description of the commutant of the Haar multishift is still an open problem even in the $L^2$-case. On the one hand, the boundedness and even the continuity of a function $f$, in general, does not guarantee the boundedness of the operator $T_f$ in $L^2$ \cite{Nik}. On the other hand, $T_f$ is bounded in the latter space under a certain minimal smoothness condition \cite{T1}. Concerning the case of a  function $f\in\mathcal{H}_{ch}^1$, the condition $f\in L^{\infty}$ is sufficient but not necessary for the operator $T_f$ to be bounded on $L^p$ for each $1<p<\infty$ \cite[Theorem 7]{AT2018}.


It is clear that the system of dilations and translations of a function $f$ forms a basis in the space $L^p$ equivalent to the Haar system if and only if the operator $T_f$ is an isomorphism of $L^p$. Moreover, as is easy to see, the inverse operator $T_f^{-1}$ commutes with the Haar multishift as well and therefore $T_f^{-1}=T_g$, where $g=T_f^{-1}h$.

\subsection{Dyadic $H^1$- and $BMO$-spaces}
\label{dyadic}

As the limiting spaces of the $L^p$-scale, $1<p<\infty$, we will consider the spaces $H^1$ and $BMO$ rather than $L^1$ and $L^\infty$   .

Let $x\in L^1=L^1[0,1]$, $\int_0^1x(t)\,dt=0$. Then, the Fourier-Haar expansion of $f$ can be rewritten by using the above introduced operators $V_0$ and $V_1$ as follows:
$$
x=\sum_{\alpha\in\mathbb{A}}\xi_{\alpha}V^{\alpha}h, \qquad \xi_{\alpha}=2^{|\alpha|}(x,V^{\alpha}h).
$$
Let $Px$ be the corresponding Paley function defined by
$$
Px(t)=\biggl(\sum_{k=0}^{\infty}|\xi(t_1,\dots,t_k)|^2\biggr)^{1/2}, \qquad t=(t_1,\dots,t_k,\dots)\in I.
$$
If $X$ is a Banach function space, $X\subset L^1$, then the Paley space $P(X)$ consists of all functions $x\in L^1$ such that $Px\in X$ and is equipped with the norm $\|x\|_{P(X)}:=\|Px\|_X$ \cite{AS}. In particular, $P(L^p)=L^p$ with equivalence of norms for every $1<p<\infty$.
On the other hand, $P(L^1):=H_d^1\varsubsetneq L^1$.

The following equivalent definition of the space $H_d^1$ can be done by using the so-called atomic decomposition (see e.g. \cite[Theorem 1.2.4]{M}). Specifically, $H_d^1$ consists of all functions $x\in L^1$, $\int_0^1x(t)\,dt=0$, admitting a representation
\begin{equation}\label{eq.at1}
x=\sum_{\alpha\in\mathbb{A}}2^{|\alpha|}V^{\alpha}x_{\alpha},
\end{equation}
where
\begin{equation}\label{eq.at2}
x_{\alpha}\in L^2, \qquad \int_0^1x_{\alpha}(t)\,dt=0, \qquad \sum_{\alpha\in\mathbb{A}}\|x_{\alpha}\|_{L^2}<\infty.
\end{equation}
Moreover, $H_d^1$ is equipped with the norm
$$
\|x\|_{H_d^1}=\inf\sum_{\alpha\in\mathbb{A}}\|x_{\alpha}\|_{L^2},
$$
where the infimum is taken over all representations \eqref{eq.at1}, \eqref{eq.at2}.

The {\it dyadic space of functions with bounded mean oscillation}, $BMO_d$, can be characterized as the set of all functions $x\in L^2$, $\int_0^1x(t)\,dt=0$, such that
\begin{equation}\label{eq.bmo}
\|x\|_{BMO_d}:=\sup_{\alpha\in\mathbb{A}}\|2^{|\alpha|}V^{\alpha*}x\|_{L^2}<\infty,
\end{equation}
where $V^{\alpha*}$ is the adjoint operator for the operator $V^{\alpha}$ in the space $\{x\in L^2:\int_0^1x(t)\,dt=0\}$. One can readily check that, for each multi-index $\alpha$, $k=|\alpha|$, and $j=\sum_{\nu=1}^k\alpha_{\nu}2^{k-\nu}$, we have
$$
2^{|\alpha|}V^{\alpha*}x(t)=x\Bigl(\frac{j+t}{2^k}\Bigr)-\int_0^1x\Bigl(\frac{j+s}{2^k}\Bigr)\,ds, \qquad t\in I.
$$

An advantage of using the norms defined by \eqref{eq.at1}, \eqref{eq.at2}, and \eqref{eq.bmo} is the fact that the bilinear form
$$
\langle x,y \rangle: =\sum_{\alpha\in\mathbb{A}}\frac{\xi_{\alpha}\eta_{\alpha}}{2^{|\alpha|}}
$$
gives precisely the canonical duality $(H_d^1)^*=BMO_d$, where $(\xi_{\alpha})_{\alpha\in\mathbb{A}}$ and $(\eta_{\alpha})_{\alpha\in\mathbb{A}}$ are the Fourier-Haar coefficients of functions $x\in H_d^1$ and $y\in BMO_d$, respectively (see e.g. \cite[Section~1.2]{M}).
Thanks to that, it is easier to calculate the norms of operators acting in the spaces $H_d^1$ and $BMO_d$ than when one uses the original (equivalent) norms
$$
\|x\|_{H_1^d}':=\int_0^1Px(t)\,dt
$$
and
$$
\|x\|_{BMO_d}':=\sup_{\alpha\in\mathbb{A}}\|2^{|\alpha|}V^{\alpha*}x\|_{L^1}
=\sup_{\alpha\in\mathbb{A}}\frac{1}{|I_{\alpha}|}\int_{I_{\alpha}}|x(t)-x_{I_{\alpha}}|\,dt,
$$
where $|I_{\alpha}|=2^{-|\alpha|}$ and $x_{I_{\alpha}}=\frac{1}{|I_{\alpha}|}\int_{I_{\alpha}}x(s)\,ds$.
Observe that the equivalence of the norms
$$
\sup_{\alpha\in\mathbb{A}}\|2^{|\alpha|}V^{\alpha*}x\|_{L^p}, \qquad 1\le p<\infty,
$$
on $BMO_d$ is a consequence of the celebrated John-Nirenberg inequality \cite{JN}.


For any function $x\in L^2$, $\int_0^1x(t)\,dt=0$, the sharp-function $x^{\sharp}$ is defined by
$$
x^{\sharp}(t):=\sup_{\alpha\,:\,I_{\alpha}\ni t}\|2^{|\alpha|}V^{\alpha*}x\|_{L^2}, \qquad t\in I,
$$
where the supremum is taken over all $\alpha\in\mathbb{A}$ such that $t\in I_{\alpha}$. It is well known that $\|x^{\sharp}\|_{L^p}\simeq\|x\|_{L^p},$ for each $2<p<\infty$ \cite[Theorem 1.2.5]{M}.
Finally, for $p=\infty$ we clearly have $\|x^{\sharp}\|_{L^{\infty}}=\|x\|_{BMO_d}$.

\subsection{Haar chaoses}
\label{Haar chaoses}

As was said in Section~\ref{Intro}, a first-order Haar chaos $f\in L^1$ is just a step-function of the form \eqref{Haar chaos}, or equivalently, a function with the Fourier-Haar expansion defined by \eqref{F-H expansion}.
A simple calculation shows then that
$$
c_k=-f(\tfrac{1}{2^k})-\sum_{j=0}^{k-1}2^{k-j-1}f(\tfrac{1}{2^j}),
$$
and conversely,
$$
f(\tfrac{1}{2^k})=-c_k+\sum_{j=0}^{k-1}c_j.
$$
As was said in Section~\ref{Intro}, a key role in the proofs will be played by the generating function $\hat{f}(z)=\sum_{k=0}^{\infty}c_kz^k$ of the sequence of the Fourier-Haar coefficients $(c_k)_{k\ge 0}$ of a  function $f\in\mathcal{H}_{ch}^1$, which will be called the {\it symbol} of $f$. 
There is the following useful link between $\hat{f}$ and the generating function $\check{f}(z)=\sum_{k=0}^{\infty}f(\tfrac{1}{2^k})z^k$ of the sequence of the values of $f$:
$$
(1-z)\check{f}(z)=(2z-1)\hat{f}(z)
$$
\cite[Lemma~5]{AT2018}. By using this equation, we can establish the following result, which makes possible to calculate (up to equivalence) the norms of an integrable first-order Haar chaos $f$ in the above spaces.

\begin{lemma}\label{l.1}
Let $1<p<\infty$. For arbitrary $f\in\mathcal{H}_{ch}^1\cap L^1$ we have
$$
\|f\|_{H_d^1}\simeq\|\hat{f}\|_{A_1^+(\mathbb{D}_{1/2})},
$$
$$
\|f\|_{L^p}\simeq\|\hat{f}\|_{A_p^+(\mathbb{D}_{2^{-1/p}})}, \quad \mbox{with a constant depending on}\;p,
$$
and
$$
\|f\|_{BMO_d}\simeq\|\hat{f}\|_{A_{\infty}^+(\mathbb{D})}.
$$
\end{lemma}
Lemma \ref{l.1} is a partial case of the next more general Lemma \ref{l.2}. To state it, we observe that an arbitrary Haar chaos of order $d\in\mathbb{N}$
$$
x=\sum_{n\in N_d}\xi_nh_{n}
$$
can be equivalently rewritten as follows:
$$
x=\sum_{k_1,\dots,k_d\ge0}\xi_{k_1,\dots,k_d}V_0^{k_1}V_1\ldots V_1V_0^{k_d}h.
$$
Here, to each $n\in N_d$, $n=2^{i_1}+\ldots+2^{i_d}$, where $i_1>i_2>\dots >i_d\ge 0$, we assign the multi-index $(0_{k_1},1,0_{k_2},1,\dots,1,0_{k_d})$ such that
$$
\begin{cases}
i_1=k_1+k_2+\ldots+k_d+d-1,\\
i_2=k_2+\ldots+k_d+d-2,\\
\vdots\\
i_d=k_d.
\end{cases}
$$
Then, the {\it symbol} of such a function $x$ is the function $\hat{x}$ in the variables $z_1,\dots,z_d$, whose coefficients coincide with the coefficients of $x$, i.e.,
$$
\hat{x}(z_1,\dots,z_d)=\sum_{k_1,\dots,k_d\ge0}\xi_{k_1,\dots,k_d}z_1^{k_1}\ldots z_d^{k_d}.
$$
Let $\mathbb{D}_R^d$ be the polydisk of radius $R$, $\mathbb{D}_R^d:=\{(z_1,\dots,z_d)\in\mathbb{C}^d:|z_1|,\dots,|z_d|<R\}$, and let $A_p^+(\mathbb{D}_R^d)$ be the space of all analytic in $\mathbb{D}_R^d$ functions
$$
u(z_1,\dots,z_d)=\sum_{k_1,\dots,k_d\ge0}a_{k_1,\dots,k_d}z_1^{k_1}\ldots z_d^{k_d}$$ such that
$$
\|u\|_{A_p^+(\mathbb{D}_R^d)}:=\|(a_{k_1,\dots,k_d}R^{k_1+\ldots+k_d})\|_{\ell_p}
=\biggl(\sum_{k_1,\dots,k_d\ge0}
(|a_{k_1,\dots,k_d}|R^{k_1+\ldots+k_d})^p\biggr)^{1/p}<\infty.
$$

\begin{lemma}\cite[Proposition 2]{AT2018}
\label{l.2}
Let $d\in\mathbb{N}$ and $1<p<\infty$. For arbitrary Haar chaos $x\in\mathcal{H}_{ch}^d\cap L^1$ it holds
$$
\|x\|_{H_d^1}\simeq2^{-d}\|\hat{x}\|_{A_1^+(\mathbb{D}^d_{1/2})},
$$
$$
\|x\|_{L^p}\simeq2^{-d/p}\|\hat{x}\|_{A_p^+(\mathbb{D}^d_{2^{-1/p}})},  \quad \mbox{with a constant depending on}\;p,
$$
and
$$
\|x\|_{BMO_d}\simeq\|\hat{x}\|_{A_{\infty}^+(\mathbb{D}^d)}.
$$
\end{lemma}


\subsection{Dual function and the biorthogonal system for a system of dilations and translations}\label{dual}

Let $f$ be a Haar chaos of order $1$, and let $\hat{f}(z)=\sum_{k=0}^{\infty}c_kz^k$ be the symbol of $f$. Suppose $c_0\neq0$ (or equivalently,  $f(1)\neq0$). Define the function $\hat{g}$ by $\hat{g}(z)=1/\hat{f}(z)$. A direct computation shows that $\hat{g}(z)=\sum_{k=0}^{\infty}d_kz^k$, where the coefficients $d_k$, $k=0,1,\dots$, are uniquely determined by the following recurrence formulae:
$$
c_0d_0=1, \qquad \sum_{j=0}^kc_{k-j}d_j=0, \quad k=1,2,\dots.
$$
We will call the function $g$, whose symbol is $\hat{g}$, {\it dual} to $f$.

Clearly, the Haar series
$$
g=\sum_{k=0}^{\infty}d_kh_{2^k}=\sum_{k=0}^{\infty}d_kV_0^kh
$$
converges for all $0<t\le1$, which however does not guarantee that $g$ is even an integrable function. It turns out that many properties of  $g$ are closely connected with the certain properties of the system of dilations and translations of $f$, in particular, with its completeness, uniform minimality and basisness. First, we present an explicit expression of functions of the biorthogonal system to the system of dilations and translations of $f$ via the Fourier-Haar coefficients of the dual function $g$.

Suppose a multi-index $\alpha\in\mathbb{A}$ contains precisely $s-1$ of "ones"\:for some $s\in\mathbb{N}$. Then, as in Subsection~\ref{generation}, we can write
$$
\alpha=(0_{k_1},1,0_{k_2},1,\dots,1,0_{k_s}), \quad \mbox{for some}\;\; k_1,\dots,k_s\ge0.
$$
Then, $|\alpha|=k_1+\ldots+k_s+s-1$ and the function
$$
g^{\alpha}:=\sum_{j=0}^{k_s}d_j2^{|\alpha|-j}V_0^{k_1}V_1\dots V_1V_0^{k_{s-1}}V_1V_0^{k_s-j}h,
$$
where $d_j$ are the Fourier-Haar coefficients of the dual function $g$, is a Haar chaos of order $s$. As a result, we obtain a system $\{g^{\alpha}\}_{\alpha\in\mathbb{A}}$.

\begin{lemma}\label{l.3}
Let $f$ be a Haar chaos of order $1$ such that $f\in L^1$ and $f(1)\neq0$.
Then, $\{g^{\alpha}\}_{\alpha\in\mathbb{A}}$ is the biorthogonal system to the system of dilations and translations $\{f_{\alpha}\}_{\alpha\in\mathbb{A}}$, i.e.,
$$
g^{\alpha}(f_{\beta})=
(f_{\beta},g^{\alpha})=\int_0^1f_{\beta}(t)g^{\alpha}(t)\,dt=\delta_{\alpha\beta}, \qquad \alpha,\beta\in\mathbb{A}.
$$
\end{lemma}

\begin{proof}
For each multi-index $\beta=(0_{j_1},1,0_{j_2},1,\dots,1,0_{j_r})$ we have
$$
f_{\beta}=\sum_{k=0}^{\infty}c_kV_0^{j_1}V_1\dots V_1V_0^{j_{r-1}}V_1V_0^{j_r+k}h.
$$
If $r\neq s$, then the Haar chaoses $g^{\alpha}$ and $f_{\beta}$ are of  different order, which implies that $g^{\alpha}(f_{\beta})=0$. Let now $r=s$. Then,
$$
g^{\alpha}(f_{\beta})=\sum_{j=0}^{k_s}\sum_{k=0}^{\infty}c_kd_j2^{|\alpha|-j}
(V_0^{j_1}V_1\dots  V_1V_0^{j_{s-1}}V_1V_0^{j_s+k}h,V_0^{k_1}V_1\dots V_1V_0^{k_{s-1}}V_1V_0^{k_s-j}h).
$$
Since the Haar functions are orthogonal, any nonzero term of this sum has indices $j_1=k_1,\dots,j_{s-1}=k_{s-1}$ and $j_s+k=k_s-j$. The latter equality is impossible if $j_s>k_s$, and hence in this case we have $g^{\alpha}(f_{\beta})=0$. Assuming that $j_s=k_s$, we get $j=k=0$ and $g^{\alpha}(f_{\beta})=c_0d_0=1$. Finally,  if $j_s<k_s$, then
$$
g^{\alpha}(f_{\beta})=\sum_{j=0}^{k_s-j_s}c_{k_s-j_s-j}d_j=0.
$$
Summarizing, we conclude that $g^{\alpha}(f_{\beta})=\delta_{\alpha\beta}$ for all $\alpha,\beta\in\mathbb{A}$.
\end{proof}

\begin{lemma}\label{l.31}
Suppose that $f\in\mathcal{H}_{ch}^1\cap L^p$, $1<p<\infty$ and  $\frac{1}{p}+\frac{1}{p'}=1$.

Then, the symbol $\hat{g}(z)$ of the dual function $g$ belongs to the space $A_{p'}^+(\mathbb{D}_{2^{-1/p}})$ if and only if
$$
\sup_{\alpha\in\mathbb{A}}\|f_{\alpha}\|_{L^p}\|g^{\alpha}\|_{L^{p'}}<\infty.
$$
\end{lemma}

\begin{proof}
We set $R_p=2^{-1/p}$ and $R_{p'}=2^{-1/p'}$. Then, using the notation, adopted just before Lemma \ref{l.3}, and applying Lemmas \ref{l.2} and \ref{l.1} to the function $g^{\alpha}\in\mathcal{H}_{ch}^s$, we obtain
$$
\|g^{\alpha}\|_{L^{p'}}\simeq2^{-s/p'}\|(d_j2^{|\alpha|-j}R_{p'}^{k_1+\ldots+k_s-j})_{j=0}^{k_s}\|_{\ell^{p'}}
\simeq2^{|\alpha|/p}\|(d_jR_p^j)_{j=0}^{k_s}\|_{\ell^{p'}}.
$$
Consequently, since $\|f_{\alpha}\|_{L^p}=2^{-|\alpha|/p}\|f\|_{L^p}$, it follows
$$
\sup_{\alpha\in\mathbb{A}}\|f_{\alpha}\|_{L^p}\|g^{\alpha}\|_{L^{p'}}
\simeq\sup_{k\ge0}\|(d_jR_p^j)_{j=0}^k\|_{\ell^{p'}}=\|\hat{g}\|_{A_{p'}^+(\mathbb{D}_{2^{-1/p}})}.
$$


\end{proof}

The last result of this section gives a condition of completeness of a system of dilations and translations $\{f_{\alpha}\}$ in $L^p$.

Let us observe that completeness is understood here up to a constant. More precisely, we say that a system $\{f_{\alpha}\}$ is {\it complete} in $L^p$, $1<p<\infty$, whenever $[V^{\alpha}f]_{L^p}=[V^{\alpha}h]_{L^p}=\{x\in L^p:\int_0^1x(t)\,dt=0\}$. Equivalently, from the conditions $y\in L^{p'}$, $\tfrac{1}{p}+\tfrac{1}{p'}=1$, and $(V^{\alpha}f,y)=0$ for all $\alpha\in\mathbb{A}$ it follows that $y\equiv const$.


\begin{lemma}\label{l.4}
Let $1<p<\infty$, and let $f\in\mathcal{H}_{ch}^1\cap L^p$.
The following conditions are equivalent:

(i) the system of dilations and translations of $f$ is complete in the space $L^p$;

(ii) for each $d\in\mathbb{N}$ the system $\{f_n\}_{n\in N_d}$ is complete in the subspace $\mathcal{H}_{ch}^d\cap L^p$;

(iii) the system $\{f_{2^k}\}_{k=0}^\infty$ is complete in the subspace $\mathcal{H}_{ch}^1\cap L^p$;

(iv) the system $\{z^k\hat{f}(z)\}_{k=0}^{\infty}$ is complete in the space $A_p^+(\mathbb{D}_{2^{-1/p}})$.
\end{lemma}


\begin{proof}
Let us prove first the equivalence $(iii)\Leftrightarrow(ii)$.

If $h\in[f_{2^k},k\ge 0]_{L^p}=[V_0^kf,k\ge 0]_{L^p}$, then $\sum_{k=0}^{n-1}a_{n,k}V_0^kf\to h$ in $L^p$ as $n\to\infty$ for some $a_{n,k}\in\mathbb{C}$. Therefore, for any integers $k_1,\dots,k_d\ge 0$ we have
$$
V_0^{k_1}V_1\ldots V_1V_0^{k_d}h=V_0^{k_1}V_1\ldots V_1V_0^{k_d}\lim_{n\to\infty}\sum_{k=0}^{n-1}a_{n,k}V_0^kf,
$$
whence $\mathcal{H}_{ch}^d\cap L^p=[V_0^{k_1}V_1\ldots V_1V_0^{k_d}h]_{L^p}=[V_0^{k_1}V_1\ldots V_1V_0^{k_d}f]_{L^p}$.
Since the converse implication $(ii)\Rightarrow(iii)$ is obvious, everything is done.

$(i)\Rightarrow(iii)$. It is well known that the average operator
$$
Qx(t)=\sum_{k=0}^\infty 2^{k+1}\int_{\frac{1}{2^{k+1}}}^{\frac{1}{2^k}}x(s)\,ds\cdot \chi_{(\tfrac{1}{2^{k+1}},\tfrac{1}{2^k}]}(t)
$$
is bounded in $L^p$ (see e.g. \cite[II.3.2]{KPS}). Moreover, one can easily see that $Q$ is a projection, whose image in $L^p$ is the subspace $\mathcal{H}_{ch}^1\cap L^p$. Therefore, if
$$
P_n(f)=\sum_{|\alpha|<n}a_n(\alpha)V^{\alpha}f, \qquad a_n(\alpha)\in\mathbb{C}, \qquad n=1,2,\dots,
$$
are polynomials, which approach the Haar function $h$ in $L^p$, that is, $\|P_n(f)-h\|_{L^p}\to 0$, then we have
$$
h=Qh=\lim_{n\to\infty}QP_n(f)=\lim_{n\to\infty}\sum_{k=0}^{n-1}a_n(0_k)V_0^kf.
$$
This implies that $[V_0^kf]_{L^p}=[V_0^kh]_{L^p}=\mathcal{H}_{ch}^1\cap L^p$, i.e., $(iii)$ is proved.

The converse implication $(iii)\Rightarrow(i)$ is almost obvious. Indeed, the subspace $[V^{\alpha}f]_{L^p}$ is invariant with respect to the action of the Haar multishift. Hence, from the fact that $h\in[V_0^kf]_{L^p}\subset[V^{\alpha}f]_{L^p}$ it follows immediately $[V^{\alpha}f]_{L^p}=[V^{\alpha}h]_{L^p}$.

Lastly, according to Lemma~\ref{l.3}, the mapping $f\mapsto\hat{f}$ is an isomorphism of the space $\mathcal{H}_{ch}^1\cap L^p$ onto $A_p^+(\mathbb{D}_{2^{-1/p}})$ satisfying the condition: $\widehat{V_0^kf}(z)=z^k\hat{f}(z)$.
This implies clearly the equivalence $(iii)\Leftrightarrow(iv)$.
\end{proof}

\begin{remark}\rm
\label{rem: vanish}
One can easily show that, for each $\lambda\in\mathbb{D}_R$, the functional $u\mapsto u(\lambda)$ is bounded on the space $A_p^+(\mathbb{D}_R)$.
Therefore, each of the conditions of Lemma \ref{l.4} fails if the symbol
$\hat{f}$ vanishes in the disk $\mathbb{D}_{2^{-1/p}}$. A simple example of such a sort is the function $f=h_2$. Clearly, we have $\hat{f}(z)=z$ and $[f_{2^k}:\,k\ge0]=[h_{2^k}:\,k\ge1]\ne [h_{2^k}:\,k\ge0]$.
\end{remark}


\begin{remark}\rm
Lemma \ref{l.4} holds also for the limiting exponents $p=1$ and $p=\infty$. In the former case, this result involves, instead of $L^p$ and $A_p^+(\mathbb{D}_{2^{-1/p}})$, the spaces $H_d^1$ and $A_{1}^+(\mathbb{D}_{1/2})$. In the latter one, $L^p$ should be replaced with the space $[V^{\alpha}h]_{BMO_d}=VMO_d$, the separable part of $BMO_d$, rather than $BMO_d$ itself. Recall that $VMO_d$ consists of all functions $x$ such that
$$
\lim_{k\to\infty}\sup_{|\alpha|=k}\|2^{|\alpha|}V^{\alpha*}x\|_{L^2}=0.
$$
Respectively, $A_p^+(\mathbb{D}_{2^{-1/p}})$ should be replaced with the  linear span $[z^k]_{A_{\infty}^+(\mathbb{D})}$ rather than the (non-separable) space $A_{\infty}^+(\mathbb{D}).$ It can be readily checked that $[z^k]_{A_{\infty}^+(\mathbb{D})}$ is the space of all analytic in $\mathbb{D}$ functions $u(z)=\sum_{k=0}^{\infty}a_kz^k$ satisfying the condition $\lim_{k\to\infty}a_k=0$.

\end{remark}


Proceeding now with the proofs, we will consider the cases $p=2$, $p>2$ and $1<p<2$ separately, to emphasize the difference in results obtained. There are two reasons to start with the situation when $p=2$. First, the Hilbertian case is the simplest one, second, "$L^2$-methods" will be intensively used later to deduce similar results for the $L^p$-spaces, $p\ne 2$, the spaces $BMO_d$ and $H_d^1$.  

\section{Case $p=2$.}\label{s.oper}

\begin{propos}\label{l.L2}
Let $f\in \mathcal{H}_{ch}^1 \cap L^2$. The following conditions are equivalent:

(i) the operator $T_f$ is bounded in the space $L^2$;

(ii) for every $d\in\mathbb{N}$ $T_f$ is bounded on the subspace of Haar chaoses $\mathcal{H}_{ch}^d \cap L^2$;

(iii) $T_f$ is bounded on the subspace of first-order Haar chaoses $\mathcal{H}_{ch}^1 \cap L^2$;

(iv) the symbol $\hat{f}(z)$ is a bounded analytic function on the disk $\mathbb{D}_{2^{-1/2}}$.
\end{propos}


\begin{proof}
First of all, since $f\in \mathcal{H}_{ch}^1$, then for a Haar chaos of  order $d\in\mathbb{N}$
$$
x=\sum_{k_1,\dots,k_d\ge0}\xi_{k_1,\dots,k_d}V_0^{k_1}V_1\dots V_1V_0^{k_d}h
$$
we have
\begin{multline*}
T_fx=\sum_{k_1,\dots,k_d\ge0}\xi_{k_1,\dots,k_d}V_0^{k_1}V_1\dots V_1V_0^{k_d}f
=\sum_{k_1,\dots,k_d\ge0}\sum_{j=0}^{\infty}\xi_{k_1,\dots,k_d}c_j V_0^{k_1}V_1\dots V_1V_0^{k_d+j}h\\
=\sum_{k_1,\dots,k_d\ge0}\biggl(\sum_{j=0}^{k_d}\xi_{k_1,\dots,k_{d-1},k_d-j}c_j\biggr)V_0^{k_1}V_1\dots V_1V_0^{k_d}h.
\end{multline*}
In the language of symbols, this can be rewritten as follows:
\begin{equation}\label{symbols}
\widehat{T_fx}(z_1,\dots,z_d)=\hat{x}(z_1,\dots,z_d)\hat{f}(z_d).
\end{equation}

In particular, from \eqref{symbols} it follows that the boundedness of the operator $T_f$ in the subspace $\mathcal{H}_{ch}^1 \cap L^2$ is equivalent to the boundedness of the multiplication operator by the function $\hat{f}(z)$ in the space $A_2^+(\mathbb{D}_{2^{-1/2}})=H^2(\mathbb{D}_{2^{-1/2}})$, i.e., to the fact that $\hat{f}\in M_2^+(\mathbb{D}_{2^{-1/2}})=H^{\infty}(\mathbb{D}_{2^{-1/2}})$. This yields the equivalence $(iii)\Leftrightarrow(iv)$.

Further, if $R=2^{-1/2}$ and $|\hat{f}(z)|\le C$, $z\in\mathbb{D}_R$, then for every $d\in\mathbb{N}$ and any $x\in\mathcal{H}_{ch}^d$, by \eqref{symbols},
\begin{multline*}
\|\widehat{T_fx}\|_{H^2(\mathbb{D}_R^d)}
=\sup_{0<r<R}
\biggl(\frac{1}{(2\pi)^d}\int_{\mathbb{T}^d}|\hat{x}(rz_1,\dots,rz_d)\hat{f}(rz_d)|^2\,d\mu\biggr)^{1/2}\\
\le C\sup_{0<r<R}\biggl(\frac{1}{(2\pi)^d}\int_{\mathbb{T}^d}|\hat{x}(rz_1,\dots,rz_d)|^2\,d\mu\biggr)^{1/2}
=C\|\hat{x}\|_{H^2(\mathbb{D}_R^d)},
\end{multline*}
where $d\mu$ is the usual Lebesgue measure on the torus $\mathbb{T}^d=\{(z_1,\dots,z_d)\in\mathbb{C}:|z_1|=\ldots=|z_d|=1\}$. Combining this together with the equality $A_2^+(\mathbb{D}_R^d)=H^2(\mathbb{D}_R^d)$, we get
$$
\|\widehat{T_fx}\|_{A_2^+(\mathbb{D}_R^d)}\le C\|\hat{x}\|_{A_2^+(\mathbb{D}_R^d)}.
$$
Therefore, since
$$
\|x\|_{L^2}=\biggl(\sum_{k_1,\dots,k_d\ge0}\frac{|\xi_{k_1,\dots,k_d}|^2}{2^{k_1+\ldots+k_d+d-1}}\biggr)^{1/2}
=2^{(d-1)/2}\|\hat{x}\|_{H^2(\mathbb{D}_R^d)},
$$
it follows
$$
\|T_fx\|_{L^2}\le C\|x\|_{L^2}, \qquad x\in\mathcal{H}_{ch}^d\cap L^2, \qquad d\in\mathbb{N}.
$$
As a result, we see that the operator $T_f$ is bounded on every subspace $\mathcal{H}_{ch}^d\cap L^2$ (uniformly in $d\in\mathbb{N}$).
Thus, the implication $(iv)\Rightarrow(ii)$ is proved.

Let us establish the implication $(ii)\Rightarrow(i)$. To this end, we will use the orthogonal expansion of any function $x\in L^2$, $\int_0^1x(t)\,dt=0$ in Haar chaoses:
$$
x=\sum_{d=1}^{\infty}x_d, \qquad x_d\in\mathcal{H}_{ch}^d.
$$
Observing that $T_fx_d\in\mathcal{H}_{ch}^d$, we see that
$$
T_fx=\sum_{d=1}^{\infty}T_fx_d
$$
is the orthogonal expansion in Haar chaoses as well. Hence,
$$
\|T_fx\|_{L^2}^2=\sum_{d=1}^{\infty}\|T_fx_d\|_{L^2}^2\le C^2\sum_{d=1}^{\infty}\|x_d\|_{L^2}^2=C^2\|x\|_{L^2}.
$$

Since the implications $(i)\Rightarrow(ii)\Rightarrow(iii)$ are obvious, the proof is completed.
\end{proof}

\begin{remark}\rm
\label{rem11}
An inspection of the proof of Proposition \ref{l.L2} shows that
$$
\|T_f\|_{L^2}=\|\hat{f}\|_{H^{\infty}(\mathbb{D}_{2^{-1/2}})}.
$$
\end{remark}

\begin{corollary}\label{c.L21}
Let $f\in\mathcal{H}_{ch}^1\cap L^2$. The following conditions are equivalent:

(i) the operator $T_f$ is an isomorphism of the space $L^2$;

(ii) $T_f$ is an isomorphism of each subspace $\mathcal{H}_{ch}^d\cap L^2$, $d\in\mathbb{N}$;

(iii) $T_f$ is an isomorphism of the subspace $\mathcal{H}_{ch}^1\cap L^2$;

(iv) the symbol $\hat{f}(z)$ is invertible in the algebra $H^{\infty}(\mathbb{D}_{2^{-1/2}})$.
\end{corollary}

\begin{proof}
As was above-mentioned, the inverse operator $T_f^{-1}$ commutes with the Haar multishift, and hence it can be represented as $T_g$ with $g=T_f^{-1}h$. Observe that $g$ is the dual function to $f$. Indeed, if
$$
g=\sum_{d=1}^{\infty}g_d, \qquad g_d\in\mathcal{H}_{ch}^d,
$$
then we have
$$
h=T_fg=\sum_{d=1}^{\infty}T_fg_d, \qquad T_fg_d\in\mathcal{H}_{ch}^d.
$$
Hence, $T_fg_d=0$ for $d\ge2$. Combining this with the fact that $Ker(T_f)=\{0\}$ (it is a consequence of the existence of a biorthogonal dual system), we conclude that $g=g_1$, i.e., $g$ is a first-order Haar chaos. Then, rewriting the equality $h=T_fg$ in the equivalent form: $1=\hat{f}(z)\hat{g}(z)$, we infer that $g$ is the dual function to $f$.

As a result, $T_f$ is an isomorphism if and only if both operators $T_f$ and $T_g$, corresponding to mutually dual functions $f,g\in \mathcal{H}_{ch}^1\cap L^2$, are bounded. So, it remains to apply Proposition \ref{l.L2}.
\end{proof}

Recalling the definition of the operator $T_f$, we can restate Corollary \ref{c.L21} as follows.

\begin{corollary}\label{c.L22}
Let $f\in \mathcal{H}_{ch}^1\cap L^2$. The following conditions are equivalent:

(i) the system of dilations and translations of $f$ is a basis in $L^2$, equivalent to the Haar system;

(ii) for every $d\in\mathbb{N}$ the subsystem $\{f_n\}_{n\in N_d}$ is a basis in the subspace $\mathcal{H}_{ch}^d\cap L^2$, equivalent to the Haar subsystem $\{h_n\}_{n\in N_d}$;

(iii) the subsystem $\{f_{2^k}\}$ is a basis in the subspace $\mathcal{H}_{ch}^1\cap L^2$, equivalent to the Haar subsystem $\{h_{2^k}\}$;

(iv) the symbol $\hat{f}(z)$ satisfies the inequality
$$
0<A\le|\hat{f}(z)|\le B<\infty, \qquad z\in\mathbb{D}_{2^{-1/2}}.
$$
\end{corollary}

\begin{corollary}\label{c.L23}
Suppose $f\in \mathcal{H}_{ch}^1\cap L^2$ such that the operator $T_f:L^2\to L^2$ is bounded. Then, the spectrum of $T_f$ coincides with the closure of the image of the symbol $\hat{f}(z)$ on the disk $\mathbb{D}_{2^{-1/2}}$, that is,
$$
\sigma(T_f)_{L^2}=\overline{\hat{f}(\mathbb{D}_{2^{-1/2}})}.
$$
Moreover, $\rho(T_f)_{L^2}=\|T_f\|_{L^2}$, where $\rho(T_f)_{L^2}$ is the spectral radius of $T_f$.
\end{corollary}

\begin{proof}
Observe that the operator $\lambda I-T_f$, $\lambda\in\mathbb{C}$, commutes with the Haar multishift and corresponds to the function $f_{\lambda}=\lambda h-f$. According to Corollary \ref{c.L21}, this operator is invertible if and only if the symbol $\hat{f}_{\lambda}(z)=\lambda-\hat{f}(z)$ is invertible in the algebra $H^{\infty}(\mathbb{D}_{2^{-1/2}})$, or equivalently,
$$
0<A\le|\lambda-\hat{f}(z)|\le B<\infty, \qquad z\in\mathbb{D}_{2^{-1/2}}.
$$
The lower estimate here just means that $\lambda\notin\overline{\hat{f}(\mathbb{D}_{2^{-1/2}})}$. In turn, the upper estimate is ensured by the boundedness of the symbol $\hat{f}(z)$ on the disk $\mathbb{D}_{2^{-1/2}}$, which is equivalent, by Proposition \ref{l.L2}, to the boundedness of $T_f$ in $L^2$.

To prove the second assertion of the corollary, it suffices to note that the operator $T_f^n$ commutes with the Haar multishift and hence has the form $T_F$, where $F=T_f^nh$. Passing to symbols, we get $\hat{F}(z)=(\hat{f}(z))^n$. Therefore, by Remark~\ref{rem11},
$$
\|T_f^n\|_{L^2}^{1/n}=\|\hat{f}^n\|_{H^{\infty}(\mathbb{D}_{2^{-1/2}})}^{1/n}=\|\hat{f}\|_{H^{\infty}(\mathbb{D}_{2^{-1/2}})}=\|T_f\|_{L^2},
$$
and the desired result follows.
\end{proof}


\section{Case $p>2$.}\label{bigger}

\begin{propos}\label{l.Lpg2}
Let $2<p<\infty$, and let $f\in \mathcal{H}_{ch}^1\cap L^p$. The following conditions are equivalent:

(i) the operator $T_f$ is bounded on the space $L^p$;

(ii) for every $d\in\mathbb{N}$ the operator $T_f$ is bounded on the subspace $\mathcal{H}_{ch}^d \cap L^p$;

(iii) $T_f$ is bounded on the subspace of the first-order Haar chaoses $\mathcal{H}_{ch}^1 \cap L^p$;

(iv) the symbol $\hat{f}(z)$ belongs to the algebra $M_p^+(\mathbb{D}_{2^{-1/p}})$.
\end{propos}

\begin{proof}
Let $f=\sum_{k=0}^\infty c_kh_{2^k}$. First, applying \eqref{symbols}
for $d=1$, $R=2^{-1/p}$ together with Lemma \ref{l.1}, we get that the estimate
$$
\|\widehat{T_fx}\|_{A_p^+(\mathbb{D}_R)}\le \|\hat{f}\|_{M_p^+(\mathbb{D}_R)}\|\hat{x}\|_{A_p^+(\mathbb{D}_R)}, \qquad x\in\mathcal{H}_{ch}^1 \cap L^p,
$$
is equivalent to the inequality
$$
\|T_fx\|_{L^p}\le C_p\|\hat{f}\|_{M_p^+(\mathbb{D}_R)}\|x\|_{L^p}, \qquad x\in\mathcal{H}_{ch}^1 \cap L^p.
$$
This observation completes the proof of the equivalence $(iii)\Leftrightarrow(iv)$.

Further, by definition, the symbol $\hat{f}$ belongs to the algebra $M_p^+(\mathbb{D}_R)$ if and only if for any sequence $(\xi_k)$ such that $(2^{-k/p}\xi_k)\in\ell^p$ we have
\begin{equation}\label{eq.convol}
\biggl(\sum_{k=0}^{\infty}2^{-k}\biggl|\sum_{j=0}^k\xi_{k-j}c_j\biggr|^p\biggr)^{1/p}\le
C\biggl(\sum_{k=0}^{\infty}2^{-k}|\xi_k|^p\biggr)^{1/p}.
\end{equation}
Let $d\in\mathbb{N}$ and $x\in\mathcal{H}_{ch}^d \cap L^p$. Substituting into inequality \eqref{eq.convol} for $(\xi_k)$ the coefficients $(\xi_{k_1,\dots,k_d})$ of the chaos $x$ with fixed indices $k_1,\dots,k_{d-1}$, raising both sides to the $p$th power and then summing over all $k_1,\dots,k_{d-1}$, we get
$$
\|\hat{x}(z_1,\dots,z_d)\hat{f}(z_d)\|_{A_p^{+}(\mathbb{D}_R^d)}\le C\|\hat{x}(z_1,\dots,z_d)\|_{A_p^{+}(\mathbb{D}_R^d)}.
$$
Therefore, from Lemma \ref{l.2} it follows that
$$
\|T_fx\|_{L^p}\le C_pC \|x\|_{L^p}, \qquad x\in\mathcal{H}_{ch}^d \cap L^p, \qquad d=1,2,\dots.
$$
Thus, the implication $(iv)\Rightarrow(ii)$ is proved. Moreover, as a sub-product, we obtain that condition $(iv)$ (or equivalently $(iii)$) ensures that the operator norms of $T_f$ on the subspaces $\mathcal{H}_{ch}^d \cap L^p$, $d\in\mathbb{N}$, are uniformly bounded.

Since the implications $(i)\Rightarrow(ii)\Rightarrow(iii)$ are obvious, it is left to prove the implication $(ii)\Rightarrow(i)$.

Suppose that the operator $T_f$ is bounded on each of the subspaces $\mathcal{H}_{ch}^d\cap L^p$ (uniformly over all $d=1,2,\dots$).
We need to show that $T_f$ is bounded on the whole space $L^p$.
To this end, in view of the equivalence of the $L^p$-norms for $p>2$ of a function $y$ and its sharp-function $y^{\sharp}$ \cite[Theorem 1.2.5]{M}, it suffices to estimate the norm $\|(T_fx)^{\sharp}\|_{L^p}$ instead of $\|T_fx\|_{L^p}$.

Let $t\in I$ be fixed. Recalling that
\begin{equation}
\label{sharp formula}
y^{\sharp}(t)=\sup_{{\alpha\,:\,I_{\alpha}\ni t}}\|2^{|\alpha|}V^{\alpha*}y\|_{L^2}
\end{equation}
(see Section~\ref{dyadic}), we will estimate the function $2^{|\alpha|}V^{\alpha*}T_fx$ for an arbitrary mean zero function $x\in L^p$ and every multi-index $\alpha=(\alpha_1,\dots,\alpha_m)$ such that $I_{\alpha}\ni t$. Firstly, by \cite[Lemma 4]{AT2018}, we have
$$
2^{|\alpha|}V^{\alpha*}T_fx=T_f2^{|\alpha|}V^{\alpha*}x
+\sum_{j=0}^{m-1}\xi(\alpha_1,\dots,\alpha_j)2^{m-j}V_{\alpha_m}^*\ldots V_{\alpha_{j+1}}^*f,
$$
where $\xi(\alpha_1,\dots,\alpha_j):=2^j(x,V_{\alpha_1}\ldots V_{\alpha_j}h)$ is the Fourier-Haar coefficient of $x$, corresponding to the multi-index $(\alpha_1,\dots,\alpha_j)$.
Assuming that $t=(0_{k_1},1,\dots,0_{k_d},1,\dots)$ is the binary notation of $t$, we see that $I_{\alpha}\ni t$ if and only if $\alpha=(0_{k_1},1,\dots,1,0_{k_{d-1}},1,0_k)$ for some $d\in\mathbb{N}$ and $0\le k\le k_d$.
In addition, since $f$ is a first-order Haar chaos, then $V_{\alpha_m}^*\ldots V_{\alpha_{j+1}}^*f=0$ whenever $\alpha_{i}=1$ for some $i=j+1,\dots,m$. Therefore, for every $\alpha$ such that $I_{\alpha}\ni t$ we have
$$
2^{|\alpha|}V^{\alpha*}T_fx=T_f2^{|\alpha|}V^{\alpha*}x+\sum_{j=0}^{k-1}\xi(0_{k_1},1,\dots,1,0_{k_{d-1}},1,0_j)2^{k-j}V_0^{(k-j)*}f.
$$
Observe that $\xi(0_{k_1},1,\dots,1,0_{k_{d-1}},1,0_j)$ is simultaneously  the Fourier-Haar coefficient of the $d$-order Haar chaos $x_d$ from the expansion of $x$ in Haar chaoses
\begin{equation}
\label{expansion in chaoses}
x=\sum_{d=1}^{\infty}x_d, \qquad x_d\in \mathcal{H}_{ch}^d \cap L^p.
\end{equation}
Hence, for $x_d$ and a multi-index $\alpha$, satisfying $I_{\alpha}\ni t$, it follows that
$$
2^{|\alpha|}V^{\alpha*}T_fx_d=T_f2^{|\alpha|}V^{\alpha*}x_d+\sum_{j=0}^{k-1}\xi(0_{k_1},1,\dots,1,0_{k_{d-1}},1,0_j)2^{k-j}V_0^{(k-j)*}f.
$$
Comparing the last equations, we conclude that for every $\alpha$, with $I_{\alpha}\ni t$, there exists $d\in\mathbb{N}$ such that
\begin{equation*}
\label{crucial1}
2^{|\alpha|}V^{\alpha*}T_fx=2^{|\alpha|}V^{\alpha*}T_fx_d-T_f2^{|\alpha|}V^{\alpha*}x_d+T_f2^{|\alpha|}V^{\alpha*}x.
\end{equation*}

Further, in view of already proved implication $(ii)\Rightarrow(iv)$, $\hat{f}(z)$ is an analytic function in the disk $\mathbb{D}_{2^{-1/p}}$. Hence, it is clearly bounded in the smaller disk $\mathbb{D}_{2^{-1/2}}$, and so, by Proposition \ref{l.L2}, the operator  $T_f$ is bounded on  $L^2$. Taking now in both sides of the latter equation first the $L^2$-norms and then the supremum over all $\alpha$ such that $I_{\alpha}\ni t$, by formula \eqref{sharp formula}, we get the following pointwise estimate:
$$
(T_fx)^{\sharp}(t)\le \sup_{d\in\mathbb{N}}(T_fx_d)^{\sharp}(t)+C\sup_{d\in\mathbb{N}}x_d^{\sharp}(t)+Cx^{\sharp}(t).
$$
Thus, in view of the above-mentioned equivalence of the $L^p$-norms for $p>2$ of a function $y$ and its sharp-function $y^{\sharp}$, to prove the boundedness of the operator $T_f$ on $L^p$ it suffices to obtain the estimates
\begin{equation}
\label{crucial}
\Bigl\|\sup_{d\in\mathbb{N}}x_d^{\sharp}\Bigr\|_{L^p}\lesssim\|x\|_{L^p}\;\;\mbox{and}\;\;\Bigl\|\sup_{d\in\mathbb{N}}(T_fx_d)^{\sharp}\Bigr\|_{L^p}\lesssim\|x\|_{L^p}.
\end{equation}


Let us introduce the Paley function $P_{ch}x$, corresponding to the expansion \eqref{expansion in chaoses}, defined by
$$
P_{ch}x(t):=\biggl(\sum_{d=1}^{\infty}|x_d(t)|^2\biggr)^{1/2}.
$$
Since the Haar system is unconditional in $L^p$, $1<p<\infty$ (see e.g. \cite[Theorem~2.c.5]{LT2}), then the series \eqref{expansion in chaoses} converges unconditionally for each $x\in L^p$. Therefore, $\|P_{ch}x\|_{L^p}\asymp \|x\|_{L^p}$ for all $x\in L^p$, which implies that
$$
\biggl(\sum_{d=1}^{\infty}\|x_d\|_{L^p}^p\biggr)^{1/p}
=\biggl\|\biggl(\sum_{d=1}|x_d(t)|^p\biggr)^{1/p}\biggr\|_{L^p}
\le\biggl\|\biggl(\sum_{d=1}|x_d(t)|^2\biggr)^{1/2}\biggr\|_{L^p}
\lesssim\|x\|_{L^p}.
$$
Combining this with the hypothesis of the uniform boundedness of $T_f$ on the subspaces $\mathcal{H}_{ch}^d\cap L^p$, $d=1,2,\dots$, we get
\begin{multline*}
\Bigl\|\sup_{d\in\mathbb{N}}(T_fx_d)^{\sharp}\Bigr\|_{L^p}\le
\biggl\|\biggl(\sum_{d=1}^{\infty}((T_fx_d)^{\sharp})^p\biggr)^{1/p}\biggr\|_{L^p}
=\biggl(\sum_{d=1}^{\infty}\|(T_fx_d)^{\sharp}\|_{L^p}^p\biggr)^{1/p}\\
\lesssim\biggl(\sum_{d=1}^{\infty}\|T_fx_d\|_{L^p}^p\biggr)^{1/p}
\lesssim\biggl(\sum_{d=1}^{\infty}\|x_d\|_{L^p}^p\biggr)^{1/p}\lesssim\|x\|_{L^p},
\end{multline*}
and the second inequality from \eqref{crucial} is proved. The proof of the first one is completely similar.

As was noticed above, from \eqref{crucial} it follows that the operator $T_f$ is bounded on $L^p$, and so the proof of the proposition is completed.


\end{proof}

\begin{remark}\rm
\label{rem:independ on p}
An easy inspection of the proof of Proposition \ref{l.Lpg2} shows that
conditions $(ii)$, $(iii)$ and $(iv)$ are equivalent for every $p\in(1,\infty)$.
\end{remark}

\begin{remark}\rm
\label{rem:oper norm}
From the proof of Proposition \ref{l.Lpg2} it follows that, for each $2<p<\infty$, there exists a constant $C_p$ such that the norm of the operator $T_f:L^p\to L^p$, where $f\in \mathcal{H}_{ch}^1\cap L^p$, satisfies  the  inequality
$$
\|\hat{f}\|_{M_p^+(\mathbb{D}_{2^{-1/p}})}\le\|T_f\|_{L^p}\le C_p\|\hat{f}\|_{M_p^+(\mathbb{D}_{2^{-1/p}})}.
$$
\end{remark}

\begin{corollary}\label{c.Lpg21}
Let $2<p<\infty$, and let $f\in \mathcal{H}_{ch}^1\cap L^p$. The following conditions are equivalent:

(i) the operator $T_f$ is an isomorphism of the space $L^p$;

(ii) $T_f$ is an isomorphism of every subspace $\mathcal{H}_{ch}^d\cap L^p$, $d\in\mathbb{N}$;

(iii) $T_f$ is an isomorphism of the subspace $\mathcal{H}_{ch}^1\cap L^p$;

(iv) the symbol $\hat{f}(z)$ is invertible in the algebra $M_p^+(\mathbb{D}_{2^{-1/p}})$.
\end{corollary}

\begin{proof}
It follows from Proposition \ref{l.Lpg2} in the same way as Corollary \ref{c.L21} follows from Proposition \ref{l.L2}.
\end{proof}

Taking into account the definition of the operator $T_f$, we can restate Corollary \ref{c.Lpg21} as follows.

\begin{corollary}\label{c.Lpg22}
Suppose that $2<p<\infty$ and $f\in \mathcal{H}_{ch}^1\cap L^p$. The following conditions are equivalent:

(i) the system of dilations and translations of $f$ is a basis in $L^p$ that is equivalent to the Haar system;

(ii) for each $d\in\mathbb{N}$ the subsystem $\{f_n\}_{n\in N_d}$ is a basis in the subspace $\mathcal{H}_{ch}^d\cap L^p$, equivalent to the Haar subsystem $\{h_n\}_{n\in N_d}$;

(iii) the subsystem $\{f_{2^k}\}_{k=0}^\infty$ is  a basis in the subspace $\mathcal{H}_{ch}^1\cap L^p$, equivalent to the Haar subsystem
$\{h_{2^k}\}_{k=0}^\infty$;

(iv) the symbol $\hat{f}(z)$ is invertible in the algebra $M_p^+(\mathbb{D}_{2^{-1/p}})$.
\end{corollary}

\begin{corollary}\label{c.Lpg23}
Suppose that the conditions of the preceding corollary hold. Then, the spectrum of the operator $T_f:L^p\to L^p$, $p>2$, is equal to the spectrum of the symbol $\hat{f}(z)$ in the algebra of multipliers $M_p^+(\mathbb{D}_{2^{-1/p}})$, i.e.,
$$
\sigma(T_f)_{L^p}=\sigma(\hat{f})_{M_p^+(\mathbb{D}_{2^{-1/p}})}.
$$
\end{corollary}

\begin{remark}\rm
From the embedding $M_p^+(\mathbb{D}_R)\subset H^{\infty}(\mathbb{D}_R)$ it follows that
$$
\overline{\hat{f}(\mathbb{D}_{2^{-1/p}})}\subset \sigma(T_f)_{L^p}.
$$
In contrast to the case $p=2$, the latter inclusion, in general, is strict. Say, the function
$$
\hat{f}(z)=\exp\Bigl(\frac{1+2^{1/p}z}{1-2^{1/p}z}\Bigr), \qquad z\in\mathbb{D}_{2^{-1/p}},
$$
is bounded together with $1/\hat{f}(z)$, but does not belong to the space $M_p^+(\mathbb{D}_{2^{-1/p}})$ (cf. \cite{Verb}). By Proposition~\ref{l.Lpg2}, this implies that the operator $T_f$ (and hence its spectrum)  is unbounded in $L^p$.
Similarly, for the spectral radius of $T_f$ we have
$$
\|\hat{f}\|_{H^{\infty}(\mathbb{D}_{2^{-1/p}})}
\le\rho(T_f)_{L^p}=\rho(\hat{f})_{M_p^+(\mathbb{D}_{2^{-1/p}})}
\le\|\hat{f}\|_{A_1^+(\mathbb{D}_{2^{-1/p}})}.
$$

\end{remark}

\begin{corollary}
If the symbol $\hat{f}(z)$ of a first-order Haar chaos $f$ is an analytic function that does not vanish in the disk $\mathbb{D}_{2^{-1/p}}$ for some $p>2$, then the system of dilations and translations of $f$ is a basis in the space $L^q$ for each $1<q<p$, equivalent to the Haar system.
\end{corollary}

\begin{proof}
One can readily check that every function, which is analytic in the disk of radius $2^{-1/p}$, belongs to the algebra of multipliers
$M_q^+(\mathbb{D}_{2^{-1/q}})$ for each $q<p$. Consequently, by Corollary \ref{c.Lpg22}, the system of dilations and translations of $f$ is a basis in $L^q$, equivalent to the Haar system, whenever  $2\le q<p$. Finally, thanks to \cite[Theorem~8]{AT2018}, we can extend this result to the range $1<q<p$.
\end{proof}

\begin{corollary}
If the symbol $\hat{f}(z)$ of a first-order Haar chaos $f$ has the absolutely converging Taylor expansion and does not vanish in the closed disk $\overline{\mathbb{D}_{2^{-1/p}}}$ for some $p\ge2$, then the system of dilations and translations of $f$ is a basis in the space $L^p$, equivalent to the Haar system.
\end{corollary}

\begin{proof}
By the well-known Wiener theorem \cite[Section~1.7]{K}, the function $1/\hat{f}(z)$ together with the function $\hat{f}(z)$ has the absolutely converging Taylor expansion. Equivalently, $1/\hat{f}(z)$ belongs to the space $A_1^+(\mathbb{D}_{2^{-1/p}})=M_1^+(\mathbb{D}_{2^{-1/p}})\subset M_p^+(\mathbb{D}_{2^{-1/p}})$ (if $p>2$) and $A_1^+(\mathbb{D}_{2^{-1/2}})\subset H^{\infty}(\mathbb{D}_{2^{-1/2}})=M_2^+(\mathbb{D}_{2^{-1/2}})$) (if $p=2$). It remains to apply Corollary \ref{c.Lpg22} and Corollary \ref{c.L22}, respectively.
\end{proof}


Next, we show that an analogue of Proposition \ref{l.Lpg2} holds also in the  case $p=\infty$ provided that, for the limiting space, we take the space $BMO_d$.

\begin{propos}\label{l.BMOd}
Let $f\in \mathcal{H}_{ch}^1\cap BMO_d$. The  following conditions are equivalent:

(i) the operator $T_f$ is bounded on the space $BMO_d$;

(ii) for every $s\in\mathbb{N}$ the operator $T_f$ is bounded on the subspace of Haar chaoses $\mathcal{H}_{ch}^s \cap BMO_d$;

(iii) $T_f$ is bounded on the subspace of first-order Haar chaoses $\mathcal{H}_{ch}^1 \cap BMO_d$;

(iv) the symbol $\hat{f}(z)$ belongs to the algebra $A_1^+(\mathbb{D})$;

(v) $f$ belongs to the space $BV(I)$ of functions of bounded variation.
\end{propos}

\begin{proof}
Clearly, we have $(i)\Rightarrow(ii)\Rightarrow(iii)$. Since $\|x\|_{BMO_d}\simeq\|\hat{x}\|_{A_{\infty}^+(\mathbb{D})}$ for $x\in\mathcal{H}_{ch}^1$ (by Lemma~\ref{l.1}), the equivalence $(iii)\Leftrightarrow(iv)$ follows, as above, from the facts that $\widehat{T_fx}(z)=\hat{f}(z)\hat{x}(z)$ and  $M_{\infty}^+(\mathbb{D})=A_1^+(\mathbb{D})$.

$(iv)\Rightarrow(i)$. By Proposition \ref{l.L2}, the operator $T_f$ is bounded in $L^2$, and hence it is well-defined on $BMO_d$. So, we need only to prove its boundedness on the latter space.

For brevity, denote by $T_0$ the operator $T_{h_2}(=T_{V_0h})$. We claim that
\begin{equation}
\label{crucial2}
\|T_0\|_{BMO_d}=1.
\end{equation}

First, for arbitrary $x\in BMO_d$, $x=\sum_{\beta\in\mathbb{A}}\xi_\beta h_\beta$, and $\alpha=(\alpha_1,\dots,\alpha_k)\in\mathbb{A}$ the commutator
$$
[2^{|\alpha|}V^{\alpha*},T_0]:=2^{|\alpha|}(V^{\alpha*}T_0-T_0V^{\alpha*})
$$
can be represented by the formula
\begin{equation*}
\label{crucial3}
[2^{|\alpha|}V^{\alpha*},T_0]x=\sum_{j=0}^{k-1}\xi(\alpha_1,\dots,\alpha_j)2^{k-j}V_{\alpha_k}^*\dots V_{\alpha_{j+1}}^*V_0h
\end{equation*}
\cite[Lemma 4]{AT2018}. Since $2^{|\beta|}V^{\beta*}V_0h=h$ for $\beta=(0)$, and $2^{|\beta|}V^{\beta*}V_0h=0$ for $\beta=(1)$ and all multi-indices $\beta$ with $|\beta|>1$, we see that the right-hand side of the last equality contains, at most, one nonzero term, corresponding to the case when $j=k-1$ and $\alpha_k=0$. Hence, denoting $\alpha':=(\alpha_1,\dots,\alpha_{k-1})$, we get
$$
2^{|\alpha|}V^{\alpha*}T_0x=\xi_{\alpha'}h+T_02^{|\alpha|}V^{\alpha*}x.
$$
Observe that the function $h$ is orthogonal to the image of $T_0$. Hence,
$$
\|2^{|\alpha|}V^{\alpha*}T_0x\|_{L^2}^2=|\xi_{\alpha'}|^2+\|T_02^{|\alpha|}V^{\alpha*}x\|_{L^2}^2.
$$
Furthermore, $\sqrt{2}T_0$ is an isometry of $L^2$. Combining this with the equality
$$
\|2^{|\alpha|}V^{\alpha*}x\|_{L^2}^2=\sum_{\beta\in\mathbb{A}}\frac{|\xi_{\alpha\beta}|^2}{2^{|\beta|}},
$$
we find
$$
\|2^{|\alpha|}V^{\alpha*}T_0x\|_{L^2}^2=|\xi_{\alpha'}|^2+\frac12\sum_{\beta\in\mathbb{A}}\frac{|\xi_{\alpha\beta}|^2}{2^{|\beta|}}
=|\xi_{\alpha'}|^2+\sum_{\beta\in\mathbb{A}}\frac{|\xi(\alpha',0,\beta)|^2}{2^{|\beta|+1}}
\le\sum_{\gamma\in\mathbb{A}}\frac{|\xi_{\alpha'\gamma}|^2}{2^{|\gamma|}},
$$
because the last sum contains the term $|\xi_{\alpha'}|^2$ (for the "empty"\:index $\gamma$), all terms of the form $\frac{|\xi(\alpha',0,\beta)|^2}{2^{|\beta|+1}}$ (for $\gamma=(0,\beta)$), and also additional terms of the form $\frac{|\xi(\alpha',1,\beta)|^2}{2^{|\beta|+1}}$ (for $\gamma=(1,\beta)$).

Combining the last inequality and \eqref{eq.bmo}, in the case when $\alpha_k=0$, we have
$$
\|2^{|\alpha|}V^{\alpha*}T_0x\|_{L^2}^2\le\|2^{|\alpha'|}V^{\alpha'*}x\|_{L^2}^2\le\|x\|_{BMO_d}^2.
$$
Otherwise, if $\alpha_k=1$ or if $\alpha$ is the "empty"\:index, then the commutator vanishes, and hence
$$
\|2^{|\alpha|}V^{\alpha*}T_0x\|_{L^2}^2=\|T_02^{|\alpha|}V^{\alpha*}x\|_{L^2}^2=\frac12\|2^{|\alpha|}V^{\alpha*}x\|_{L^2}^2\le\frac12\|x\|_{BMO_d}^2.
$$
Thus, for all $x\in BMO_d$ we have
$$
\|T_0x\|_{BMO_d}=\sup_{\alpha\in\mathbb{A}}\|2^{|\alpha|}V^{\alpha*}T_0x\|_{L^2}\le\|x\|_{BMO_d}.
$$
Since $\|T_0h\|_{BMO_d}=\|h\|_{BMO_d}$, equality \eqref{crucial2} is proved.

Further, set $f=\sum_{k=0}^{\infty}c_kh_{2^k}$. Then, by the hypothesis of condition (iv), we have $\|\hat{f}\|_{A_1^+(\mathbb{D})}:=\sum_{k=0}^{\infty}|c_k|<\infty$. Therefore, from \eqref{crucial2} it follows that
$$
\sum_{k=0}^{\infty}|c_k|\|T_0^k\|_{BMO_d}<\infty.
$$
i.e., the operator series $\sum_{k=0}^{\infty}c_kT_0^k$ absolutely converges. On the other hand, by the definition of $f$, this series should converge to $T_f$. As a result, $T_f$ is bounded in $BMO_d$, and so the proof of implication $(iv)\Rightarrow(i)$ is completed.

Finally, the equivalence $(iv)\Leftrightarrow(v)$ is an immediate consequence of the easy observation that
$$
\sum_{k=0}^{\infty}|f(\tfrac{1}{2^{k}})-f(\tfrac{1}{2^{k+1}})|=\sum_{k=0}^{\infty}|2c_k-c_{k+1}|\simeq\sum_{k=0}^{\infty}|c_k|.
$$
\end{proof}

\begin{corollary}\label{c.BMOd1}
Let $f\in \mathcal{H}_{ch}^1\cap BMO_d$. The  following conditions are equivalent:

(i) the operator $T_f$ is an isomorphism of the space $BMO_d$;

(ii) for every $s\in\mathbb{N}$ the operator $T_f$ is an isomorphism of the subspace of Haar chaoses $\mathcal{H}_{ch}^s\cap BMO_d$;

(iii) $T_f$ is an isomorphism of the subspace of first-order Haar chaoses $\mathcal{H}_{ch}^1\cap BMO_d$;

(iv) the symbol $\hat{f}(z)$ is invertible in the algebra $A_1^+(\mathbb{D})$.
\end{corollary}

\begin{proof}
The implications $(i)\Rightarrow(ii)\Rightarrow(iii)$ are obvious. From Proposition~\ref{l.BMOd} it follows that $(iii)$ implies $(iv)$.

$(iv)\Rightarrow(i)$. It suffices to use the following easy $L^2$-argument. By condition, the symbols $\hat{f}(z)$ and $\hat{g}(z)$, where $g$ is the dual function for $f$, belong to the algebra $A_1^+(\mathbb{D})$. Therefore, by Proposition~\ref{l.BMOd}, both operators $T_f$ and $T_g$ are bounded in $BMO_d$. Let $y\in BMO_d$ be arbitrary. Then $T_gy\in BMO_d$ and $T_fT_gy\in BMO_d$. But we know that $T_g$ is the inverse operator for $T_f$ in $L^2$. Thus, $T_fT_gy=y$, which implies that $T_f$ is an isomorphism of $BMO_d$.
\end{proof}

\begin{remark}\rm
The condition $\hat{f}\in A_1^+(\mathbb{D})$ ensures the convergence of the series $f=\sum_{k=0}^{\infty}c_kh_{2^k}$ in the $BMO_d$-norm, which is equivalent to the fact that $f$ belongs to the space $VMO_d=[h_n]_{BMO_d}$. Moreover, the operator $T_f$ acts boundedly in $VMO_d$.
\end{remark}

\begin{corollary}\label{c.BMOd2}
Let $f\in \mathcal{H}_{ch}^1 \cap VMO_d$. The  following conditions are equivalent:

(i) the system of dilations and translations of $f$ is a basis in the space $VMO_d$, equivalent to the Haar system;

(ii) for every $s\in\mathbb{N}$ the subsystem $\{f_n\}_{n\in N_s}$ is a basis in the subspace $\mathcal{H}_{ch}^s\cap VMO_d$, equivalent to the Haar subsystem $\{h_n\}_{n\in N_s}$;

(iii) the subsystem $\{f_{2^k}\}_{k=0}^\infty$ is a basis of the subspace $\mathcal{H}_{ch}^1 \cap VMO_d$, equivalent to the Haar subsystem
$\{h_{2^k}\}_{k=0}^\infty$;

(iv) the symbol $\hat{f}(z)$ has the absolutely converging Taylor expansion and does not vanish in the disk $\overline{\mathbb{D}}$;

(v) the function $f$ and its dual function $g$ belong to the space $BV(I)$.
\end{corollary}

\begin{remark}\rm
If the operator $T_f$ is bounded in $BMO_d$, then we have
$$
\rho(T_f)_{BMO_d}=\|T_f\|_{BMO_d}=\|\hat{f}\|_{A_1^+(\mathbb{D})}\quad\mbox{and}\quad \sigma(T_f)_{BMO_d}=\hat{f}(\overline{\mathbb{D}}).
$$
\end{remark}

\section{Case $p<2$}
\label{less}

Unlike the preceding sections, we start here with studying the limiting situation $p=1$, i.e., we examine first the behaviour of the operator $T_f$ in the space $H_d^1$.

\begin{propos}\label{l.Hd}
For every $f\in\mathcal{H}_{ch}^1$ the following conditions are equivalent:

(a) $f\in H_d^1$ and the operator $T_f$ is bounded in $H_d^1$;

(b) $f\in L^2$ and the operator $T_f$ is bounded in $L^2$.
\end{propos}

\begin{proof}
Let us show that the implication $(b)\Rightarrow(a)$ holds for an arbitrary (not necessarily being a first-order Haar chaos) $f\in L^2$.

By using the atomic decomposition of the space $H_d^1$ (see Section~\ref{dyadic}), let us represent a function $x\in H_d^1$ as follows:
$$
x=\sum_{\alpha\in\mathbb{A}}2^{|\alpha|}V^{\alpha}x_{\alpha}, \quad\mbox{where}\; x_{\alpha}\;\mbox{are mean zero and}\;\; \sum_{\alpha\in\mathbb{A}}\|x_{\alpha}\|_{L^2}<\infty.
$$
Then, since $T_f$ commutes with the Haar multishift, we have
$$
T_fx=\sum_{\alpha\in\mathbb{A}}2^{|\alpha|}V^{\alpha}T_fx_{\alpha},
$$
whence
$$
\|T_fx\|_{H_d^1}\le\sum_{\alpha\in\mathbb{A}}\|T_fx_{\alpha}\|_{L^2}\le\|T_f\|_{L^2}\sum_{\alpha\in\mathbb{A}}\|x_{\alpha}\|_{L^2}.
$$
As a result, taking the infimum over all suitable representations of $x$, we get the desired estimate
$$
\|T_f\|_{H_d^1}\le\|T_f\|_{L^2}.
$$

$(a)\Rightarrow(b)$. Let now $f\in\mathcal{H}_{ch}^1$, $f=\sum_{k=0}^\infty c_kh_{2^k}$. We show first that from the boundedness of the operator $T_f$ in the space $H_d^1$ it follows $f\in L^2$.

For a fixed $n\in\mathbb{N}$, we claim that the dyadic intervals of the form
$$
I(0_{k_1},1,\ldots,0_{k_{s-1}},1,0_n), \qquad k_1,\dots,k_{s-1}<n, \qquad s\in\mathbb{N},
$$
are pairwise disjoint. On the contrary, assume that
$$
I(0_{j_1},1,\dots,0_{j_{r-1}},1,0_n)\subset I(0_{k_1},1,\dots,0_{k_{s-1}},1,0_n)
$$
for some $j_1,\dots,j_{r-1},k_1,\dots,k_{s-1}<n$, $r,s\in\mathbb{N}$. This inclusion implies that there is $\alpha\in\mathbb{A}$ such that
$$
(0_{j_1},1,\dots,0_{j_{r-1}},1,0_n)=(0_{k_1},1,\dots,0_{k_{s-1}},1,0_n,\alpha).
$$
Clearly, $r\ge s$ and $j_1=k_1$, \dots, $j_{s-1}=k_{s-1}$. If $r=s$, the index $\alpha$ is "empty"\:and so the intervals coincide. Otherwise, assuming that $r>s$, we have
$$
(0_{j_s},1,\dots,0_{j_{r-1}},1,0_n)=(0_n,\alpha).
$$
Hence, $j_s\ge n$, which contradicts the definition of the intervals.

Thus, our claim is proved, and therefore the function $x_0$ defined by
$$
x_0:=\sum_{s=1}^{\infty}\sum_{(k_1,\dots,k_s)\in K^s(n)}V_0^{k_1}V_1\ldots V_1V_0^{k_s}h,
$$
where
$$
K^s(n)=\{(k_1,\dots,k_s):0\le k_1,\dots,k_{s-1}<n,k_s=n\}, \qquad s\in\mathbb{N},
$$
is a sum of pairwise disjoint Haar functions, and in consequence $\|x_0\|_{BMO_d}\le1$ (in fact, $\|x_0\|_{BMO_d}=1$, because the union of supports of involved Haar functions coincides with the whole interval $I$).


Further, a direct calculation shows that for a function $y\in BMO_d=(H_d^1)^*$ such that
$$
y=\sum_{s=1}^{\infty}\sum_{k_1,\dots,k_s\ge0}\eta_{k_1,\dots,k_s}V_0^{k_1}V_1\ldots V_1V_0^{k_s}h
$$
we have
$$
T_f^*y=\sum_{s=1}^{\infty}\sum_{k_1,\dots,k_s\ge0}
\biggl(\sum_{j=0}^{\infty}\frac{c_j\eta_{k_1,\dots,k_{s-1},k_s+j}}{2^j}\biggr)V_0^{k_1}V_1\ldots V_1V_0^{k_s}h.
$$
In particular, substituting into this formula $x_0$ for $y$, we get
$$
T_f^*x_0=\sum_{s=1}^{\infty}\sum_{k_1=0}^{n-1}\ldots\sum_{k_{s-1}=0}^{n-1}\sum_{k_s=0}^n
\frac{c_{n-k_s}}{2^{n-k_s}}V_0^{k_1}V_1\ldots V_1V_0^{k_s}h.
$$
Hence,
$$
\|T_f^*x_0\|^2_{L^2}=\sum_{s=1}^{\infty}\sum_{k_1=0}^{n-1}\ldots\sum_{k_{s-1}=0}^{n-1}\sum_{k_s=0}^n
\frac{1}{2^{k_1+\ldots+k_s+s-1}}\Bigl|\frac{c_{n-k_s}}{2^{n-k_s}}\Bigr|^2
=\sum_{j=0}^n\frac{|c_j|^2}{2^j},
$$
because
$$
\sum_{s=1}^{\infty}\frac{1}{2^{s-1}}\biggl(\sum_{k=0}^{n-1}\frac{1}{2^k}\biggr)^{s-1}\frac{1}{2^n}
=\frac{1}{2^n}\sum_{s=1}^{\infty}\Bigl(1-\frac{1}{2^n}\Bigr)^{s-1}=1.
$$
Thus, taking into account that $BMO_d\subset L^2$ with constant $1$ and $T_f$ is bounded in $H_d^1$, we have
$$
\biggl(\sum_{j=0}^n\frac{|c_j|^2}{2^j}\biggr)^{1/2}=\|T_f^*x_0\|_{L^2}\le\|T_f^*x_0\|_{BMO_d}\le\|T_f^*\|_{BMO_d}=\|T_f\|_{H_d^1}<\infty.
$$
Since $n\in\mathbb{N}$ is arbitrary, then from Lemma \ref{l.1} it follows that $f\in L^2$.

Next, we establish the boundedness of the operator $T_f:L^2\to L^2$. Let
$$
f_r:=\sum_{k=0}^{\infty}c_kr^kV_0^kh, \qquad 0<r<1.
$$
Passing to symbols, we have $\hat{f}_r(z)=\hat{f}(rz)$. Since $f\in L^2$, then its symbol $\hat{f}(z)$ is an analytic function in the disk $\mathbb{D}_{2^{-1/2}}$, and hence each symbol $\hat{f}_r(z)$ is a bounded analytic function in this disk. Then, by Proposition \ref{l.L2}, the operator $T_{f_r}$ is bounded in $L^2$ for every $0<r<1$. Let us prove the following inclusion for the spectra of $T_{f_r}$ and $T_{f}$ in the spaces $L^2$ and $H_d^1$ respectively:
\begin{equation}\label{eq.sp}
\sigma(T_{f_r})_{L^2}\subset\sigma(T_f)_{H_d^1}, \qquad 0<r<1.
\end{equation}


Indeed, assume that the operator $\lambda I-T_f$ is invertible in the space $H_d^1$. Observe that the inverse operator $(\lambda I-T_f)^{-1}$ commutes with the Haar multishift and can be representable as $T_{g_{\lambda}}$, with $g_{\lambda}=(\lambda I-T_f)^{-1}h\in H_d^1$. As was shown, this implies that $g_{\lambda}\in L^2$. Moreover, since  the symbol $\hat{g}_{\lambda}(z)=(\lambda-\hat{f}(z))^{-1}$ is an analytic function in the disk $\mathbb{D}_{2^{-1/2}}$, we have $\lambda\notin\hat{f}(\mathbb{D}_{2^{-1/2}})$.
All the more, so $\lambda\notin\hat{f}_r(\overline{\mathbb{D}_{2^{-1/2}}})=\overline{\hat{f}_r(\mathbb{D}_{2^{-1/2}})}=\sigma(T_{f_r})_{L^2}$ (see Corollary~\ref{c.L23}), i.e., the operator $\lambda I-T_{f_r}$ is invertible in the space $L^2$. Clearly, this completes the proof of \eqref{eq.sp}.

Further, appealing once more to Corollary~\ref{c.L23}, we see that the spectral radius of an operator of the form $T_{g}$ in $L^2$ is equal to its $L^2$-norm. Therefore, from \eqref{eq.sp} it follows that
$$
\|T_{f_r}\|_{L^2}=\rho(T_{f_r})_{L^2}\le\rho(T_f)_{H_d^1}=C<\infty, \qquad 0<r<1,
$$
i.e., the operators $T_{f_r}:L^2\to L^2$ are uniformly bounded. In view of the equality $\|T_{f_r}\|_{L^2}=\|\hat{f}_r\|_{H^{\infty}(\mathbb{D}_{2^{-1/2}})}$ by Remark~\ref{rem:oper norm}, this yields that the functions $\hat{f}_r$, $0<r<1$, are uniformly bounded in the disk $\mathbb{D}_{2^{-1/2}}$. Thus, $\hat{f}$ is bounded in $\mathbb{D}_{2^{-1/2}}$, and so, by Proposition \ref{l.L2}, the operator $T_f:L^2\to L^2$ is bounded.
\end{proof}

\begin{remark}\rm
In the proof of the implication $(a)\Rightarrow(b)$ of Proposition \ref{l.Hd} we make use of a construction of a special function $x_0$ similar to that one can find in the papers \cite{Sch}, \cite{S}, and \cite{SS}, which are devoted to studying rearrangements of the Haar system in function spaces.
\end{remark}

\begin{remark}\rm
As was showed in \cite[Theorems~4 and 8]{AT2018}, from the boundedness of the operator $T_f$ on $L^p$ for some $1<p<\infty$ it follows that $T_f$ is bounded both in $H_d^1$ and in $L^q$ for all $1<q<p$. Hence, the conditions $(a)$ and $(b)$ of Proposition \ref{l.Hd} are equivalent to each of the following:

(c) there is $1<p<2$ such that $f\in L^p$  and the operator $T_f$ is bounded in $L^p$;

(d) $f\in\bigcap_{1<p<2}L^p$ and $T_f$ is bounded in $L^p$ for all $1<p<2$.
\end{remark}

\begin{corollary}\label{c.Hd1}
Let $f\in\mathcal{H}_{ch}^1$. The following conditions are equivalent:

(a) $f\in H_d^1$ and the operator $T_f$ is an isomorphism of the space $H_d^1$;

(b) there is $1<p<2$ such that $f\in L^p$ and $T_f$ is an isomorphism of the space $L^p$;

(c) $f\in\bigcap_{1<p<2}L^p$ and $T_f$ is an isomorphism of $L^p$ for all $1<p<2$;

(d) $f\in L^2$ and $T_f$ is an isomorphism of $L^2$.
\end{corollary}

\begin{remark}\rm
\label{spectrum stab}
The results obtained indicate a substantial difference in the spectral properties of the operator $T_f$ in the $L^p$-spaces, corresponding to a first-order Haar chaos $f$, for $1<p<2$ and $2<p<\infty$. Namely, if $p>2$, then the spectrum of $T_f$ grows in increasing $p$, what is demonstrated by the following inclusions:
$$
\overline{\hat{f}(\mathbb{D}_{2^{-1/2}})}=\sigma(T_f)_{L^2}\subset\sigma(T_f)_{L^p}\subset\sigma(T_f)_{BMO_d}=\hat{f}(\overline{\mathbb{D}}),
\qquad 2<p<\infty.
$$
In contrast to that, in the case when $1<p<2$, we have stabilization of the spectrum:
$$
\sigma(T_f)_{H_d^1}=\sigma(T_f)_{L^p}=\sigma(T_f)_{L^2}.
$$

Notice also, in passing, that, according to \cite{AT2018a}, for a first-order Rademacher chaos $f=\sum_{k=0}^{\infty}a_kr_k$, $\sum_{k=0}^{\infty}|a_k|^2<\infty$ ($r_k$ are the Rademacher functions), the spectrum $\sigma(T_f)_{L^p}$ is independent of $p\in(1,\infty)$.
\end{remark}

\begin{corollary}\label{c.Hd2}
Let $f\in\mathcal{H}_{ch}^1$, and let $\{f_n\}_{n=1}^\infty$ be the system of dilations and translations of $f$. The following conditions are equivalent:

(a) $f\in H_d^1$ and $\{f_n\}$ is a basis in $H_d^1$, equivalent to the Haar system;

(b) there is $1<p<2$ such that $f\in L^p$ and $\{f_n\}$ is a basis in $L^p$, equivalent to the Haar system;

(c) $f\in\bigcap_{1<p<2}L^p$ and $\{f_n\}$ is a basis in $L^p$ for all $1<p<2$, equivalent to the Haar system;

(d) $f\in L^2$ and $\{f_n\}$ is a basis in $L^2$, equivalent to the Haar system.
\end{corollary}

\section{Proofs of the main results.}
\label{Proofs}

Observe first that Theorem \ref{t.1} is an immediate consequence of Corollaries \ref{c.L22} and \ref{c.Lpg22}, while Theorem \ref{t.2} follows from Corollary \ref{c.Hd2}.

\begin{proof}[Proof of Theorem \ref{t.3}]
Let $W^{\alpha}f$, $\alpha\in\mathbb{A}$, be the affine Walsh system generated by the given function $f$, that is,
$$
W^{\alpha}f=\sum_{|\beta|=k}(-1)^{(\alpha,\beta)}V^{\beta}f, \qquad \alpha\in\mathbb{A}, \qquad |\alpha|=k,
$$
where
$$
(\alpha,\beta):=\sum_{\nu=1}^k\alpha_{\nu}\beta_{\nu}, \qquad \alpha,\beta\in\mathbb{A}, \qquad |\alpha|=|\beta|=k.
$$
According to \cite[Theorem 7]{AT2018a}, the conditions of the theorem ensure that $\{W^{\alpha}f\}_{\alpha\in\mathbb{A}}$ is a basis in $L^q$ for all $1<q<p$. Then, clearly, the system of dilations and translations of $f$ is a basis in $L^q$, $1<q<p$, as well. Moreover, by Theorems \ref{t.1} and \ref{t.2}, it is not equivalent to the Haar system.
\end{proof}

\begin{proof}[Proof of Theorem \ref{t.4}]
Each function $x\in L^p$ can be represented in the form $x=\sum_{d=1}^\infty x_d$, where $x_d\in\mathcal{H}_{ch}^d \cap L^p$ (convergence in $L^p$).
Let $g=\sum_{k=0}^\infty d_kV_0^kh$ be the dual function for the given first-order Haar chaos $f$. By hypothesis, both functions $\hat{f}(z)$ and $1/\hat{f}(z)$ belong to the space $M_p^+(\mathbb{D}_{2^{-1/p}})$. Therefore, by Proposition \ref{l.Lpg2} and Remark \ref{rem:independ on p}, for every $1<p<\infty$ the operators $T_f$ and $T_g$ are bounded on the subspaces $\mathcal{H}_{ch}^d \cap L^p$, $d=1,2,\dots$. This implies that
\begin{multline*}
x_d=T_fT_gx_d=T_fT_g\sum_{k_1,\dots,k_d\ge0}\xi_{k_1,\dots,k_d}V_0^{k_1}V_1\ldots V_1V_0^{k_d}h\\
=T_f\sum_{k_1,\dots,k_d\ge0}\biggl(\sum_{j=0}^{k_s}\xi_{k_1,\dots,k_{s-1},k_s-j}d_j\biggr)V_0^{k_1}V_1\ldots V_1V_0^{k_d}h\\
=\sum_{k_1,\dots,k_d\ge0}\biggl(\sum_{j=0}^{k_s}\xi_{k_1,\dots,k_{s-1},k_s-j}d_j\biggr)V_0^{k_1}V_1\ldots V_1V_0^{k_d}f.
\end{multline*}
As a result, for every $x\in L^p$ we obtain the expansion of the form \eqref{exp via chaoses}. Hence, the proof of the theorem is completed.
\end{proof}


\section{Some final remarks and examples}
\label{Examples}

Let $T:=T_{h_2}=T_{V_0h}$. As was said above, this operator commutes with the Haar multishift, and from the connection between a first-order Haar chaos $f=\sum_{k=0}^{\infty}c_kh_{2^k}$ and its symbol $\hat{f}(z)=\sum_{k=0}^{\infty}c_kz^k$ it follows the following representation of the operator $T_f$:
\begin{equation}\label{operator}
T_f=\hat{f}(T):=\sum_{k=0}^{\infty}c_kT^k.
\end{equation}
By Corollaries~\ref{c.L23}, \ref{c.Lpg23} and Remark~\ref{spectrum stab},  the spectrum $\sigma(T)_{L^p}$ of the operator $T:L^p\to L^p$ is defined by
$$
\sigma(T)_{L^p}=\overline{\mathbb{D}}_{R_p}, \qquad\mbox{where}\;\;
R_p=
\begin{cases}
2^{-1/2}, & 1<p\le2,\\
2^{-1/p}, & 2\le p<\infty.
\end{cases}
$$
Combining this together with \eqref{operator} and applying the spectral mapping theorem, we get
$$
\sigma(T_f)_{L^p}=\sigma(\hat{f}(T))_{L^p}=\hat{f}(\sigma(T)_{L^p})=\hat{f}(\overline{\mathbb{D}}_{R_p}),
$$
provided that the symbol $\hat{f}(z)$ is an analytic function in some disk $\mathbb{D}_R$ of radius $R>R_p$. Let us consider first the following simplest case of such a situation.

\begin{example}\rm
Let a symbol $\hat{f}$ be just an algebraic polynomial, i.e., $\hat{f}(z)=\sum_{k=0}^nc_kz^k$. Then the operator $T_f$ is being an isomorphism of the space $L^p$ (respectively, the system of dilations and translations $(f_n)_{n=1}^\infty$ of $f$ is a basis in $L^p$, equivalent to the Haar system) if and only if the polynomial $\hat{f}(z)$ does not vanish in the closed disk $\overline{\mathbb{D}}_{R_p}$. Moreover, assume that $z_0$ is the smallest  (in modulus) zero of the polynomial $\hat{f}(z)$. Then the basisness properties of the system $(f_n)$ in the scale of $ L^p$-spaces, $1<p<\infty$, are completely described as follows:

(a) if $|z_0|\le\frac12$, then $\{f_n\}$ fails to be a basis in each of the spaces $L^p$, $1<p<\infty$;

(b) if $|z_0|=2^{-1/p_0}$, $1<p_0\le2$, then $\{f_n\}$ is a basis in $L^p$ if and only if $1<p<p_0$. For such a $p$ the system $\{f_n\}$ is not equivalent to the Haar system in $L^p$;

(c) if $|z_0|=2^{-1/p_0}$, $2<p_0<\infty$, then $\{f_n\}$ is a basis in $L^p$ if and only if $1<p<p_0$. For such a $p$ the system $\{f_n\}$ is equivalent to the Haar system in $L^p$;

(d) if $|z_0|\ge1$, then $\{f_n\}$ is a basis in $L^p$ for every $1<p<\infty$, equivalent to the Haar system.

\end{example}

In fact, the same characterization as for polynomials holds also for an arbitrary symbol $\hat{f}(z)$ with "excessive"\:analyticity, i.e., in the case when its smallest  (in modulus) zero $z_0$ lies inside the disk of analyticity of $\hat{f}(z)$. Unlike to that, in the next example the smallest  (in modulus) zero lies on the boundary of the disk of analyticity.

\begin{example}\rm
Let $1<p<\infty$ and $0<\theta<1-\frac1p$. Consider the first-order Haar chaos $f$ with the symbol
$$
\hat{f}(z)=(1-2^{1/p}z)^{\theta}=1-\sum_{k=1}^{\infty}c_k2^{k/p}z^k.$$
Then, the symbol of the dual function $g$ is defined by
$$
\hat{g}(z)=(1-2^{1/p}z)^{-\theta}=1+\sum_{k=1}^{\infty}d_k2^{k/p}z^k.$$
As an easy calculation shows, we have $c_k\simeq\frac{1}{(k+1)^{1+\theta}}$ and $d_k\simeq\frac{1}{(k+1)^{1-\theta}}$ for large $k\in\mathbb{N}$. Therefore, by Lemma~\ref{l.1}, $f,g\in L^p$. Moreover, $\hat{f}\in M_p^+(\mathbb{D}_{2^{-1/p}})$, because the Taylor expansion of $\hat{f}$ absolutely converges in the disk of radius $2^{-1/p}$, while $\hat{g}\notin M_p^+(\mathbb{D}_{2^{-1/p}})$, because the function $\hat{g}$ is unbounded in this disk. Putting all together, we see that, for every $1<p<\infty$, the system of dilations and translations of $f$ is not equivalent in  $L^p$ to the Haar system.

Furthermore, if $p>2$ and $\theta\ge\frac1p$, then $\hat{g}\notin A_{p'}(\mathbb{D}_{2^{-1/p}})$ and so, by Lemma~\ref{l.31}, the system of dilations and translations $\{f_n\}$ of $f$ fails to be uniformly minimal and hence, all the more, to be a basis in $L^p$. In contrast to that, for $p=2$ and $\theta<\frac12$,  the system $\{f_n\}$ is a basis in chaoses
in $L^2$. Indeed, from the well-known Babenko's result \cite{Bab} (see also \cite[p.~1638]{Tz-Handbook}), it follows that the system $\{e^{int}|t|^{\theta}\}_{n\in\mathbb{Z}}$ is a basis in the space $L^2(\mathbb{T})$ (not being a Riesz basis if $\theta\ne 0$) for each $\theta\in (-1/2,1/2)$. Hence, taking into account that $|\hat{f}(z)|\simeq|\arg z|^{\theta}$, $|z|=2^{-1/2}$, we conclude that the system $\{z^n\hat{f}(z)\}_{n\ge0}$ is a basis in $H^2(\mathbb{D}_{2^{-1/2}})$, or equivalently, $\{f_{2^k}\}_{k\ge0}$ is a basis in $\mathcal{H}_{ch}^1\cap L^2$. From the results obtained then it follows that $\{f_n\}_{n\in N_d}$ is a basis in the subspace $\mathcal{H}_{ch}^d\cap L^2$ for each $d=1,2,\dots$. Repeating now the proof of Theorem~\ref{t.4}, we come to the desired result.
\end{example}


\end{document}